\documentclass[a4paper,11pt]{amsart}
\usepackage{amsthm}
\usepackage{amssymb}
\usepackage{amsmath}
\usepackage{stmaryrd}
\usepackage{nicefrac}
\usepackage{bbm}
\usepackage{dsfont}
\usepackage{fullpage}

\input xy
\xyoption{all} 
\usepackage{tikz-cd}
\usepackage{amsthm}
\usepackage{amssymb}
\usepackage{amsmath}
\usepackage[shortlabels]{enumitem}

\input xy
\xyoption{all} 

\newtheorem{theorem}{Theorem}
\newtheorem*{theorem*}{Theorem}
\newtheorem*{remark*}{Remark}

\newtheorem*{example*}{Example}

\newtheorem*{examples*}{Examples}
\newtheorem*{exercise*}{Exercise}

\newtheorem{definition}[theorem]{Definition}
\newtheorem*{definition*}{Definition}
\newtheorem{lemma}[theorem]{Lemma}

\newtheorem{proposition}[theorem]{Proposition}
\newtheorem{remark}[theorem]{Remark}
\newtheorem{cor}[theorem]{Corollary}
\newtheorem*{cor*}{Corollary}

\newtheorem*{conjecture*}{Conjecture}
\numberwithin{theorem}{section}

\renewcommand\iff{%
\ifmmode\text{ if and only if }%
\else if and only if \fi}

\renewcommand{\and}{\wedge}

\renewcommand{\phi}{\varphi}

\newcommand{\Mod}{\textnormal{Mod}}
\renewcommand{\mod}{\text{mod}}

\newcommand{\Ab}{\text{Ab}}

\newcommand{\rad}{\textnormal{rad}}

\newcommand{\Ass}{\textnormal{Ass}}
\newcommand{\Div}{\textnormal{Div}}
\newcommand{\mdim}{\textnormal{mdim}}

\newcommand{\Att}{\textnormal{Att}}

\newcommand{\Hom}{\textnormal{Hom}}

\newcommand{\spec}{\textnormal{Spec}}
\newcommand{\Spec}{\textnormal{Spec}}
\newcommand{\zg}{\textnormal{Zg}}
\newcommand{\Zg}{\textnormal{Zg}}

\newcommand{\pinj}{\textnormal{pinj}}

\newcommand{\mcal}[1]{\mathcal{#1}}

\newcommand{\mfrak}[1]{\mathfrak{#1}}
\newcommand{\st}{\ \vert \ }

\newcommand{\pp}{\textnormal{pp}}

\newcommand{\R}{\mathbb{R}}

\newcommand{\N}{\mathbb{N}}
\newcommand{\Z}{\mathbb{Z}}

\newcommand{\Ideals}{\texttt{Ideals}}
\newcommand{\Filters}{\texttt{Filters}}
\newcommand{\val}{\textnormal{v}}
\newcommand{\fdrk}{\textnormal{fdrk\,}}
\newcommand{\rk}{\textnormal{rk\,}}
\newcommand{\br}{\textnormal{br\,}}
\newcommand{\CBrank}{\textnormal{CBrank}}
\newcommand{\Ord}{\textnormal{Ord}}
\newcommand{\Supp}{\textnormal{Supp}}
\newcommand{\supp}{\textnormal{Supp}}

\renewcommand{\rk}{\textnormal{rk}}
\newcommand{\nf}{\nicefrac}

\newcommand{\two}{\mathbbm{2}}
\newcommand{\Ch}{\mathbbm{Ch}}

\newif\iflorna\lornafalse

\usepackage[pdftex,linktocpage]{hyperref}

\title{Dimensions on Lattice Ordered Abelian Groups and Model Theory of Modules over Pr\"ufer Domains.}
\author{Lorna Gregory}
\address{Dipartimento di Matematica e Fisica, Universit\`a degli Studi della Campania ``Luigi
Vanvitelli'', Viale Lincoln 5, 81100 Caserta, Italy}
\email{Lorna.Gregory@gmail.com}

\thanks{The author was supported by PRIN 2017-Mathematical Logic: models, sets, computability.}



\begin{document}

\begin{abstract}
We prove a transfer theorem which, when combined with the Jaffard-Kaplansky-Ohm Theorem, allows results in model theory of modules over B\'ezout domains to be translated into results over Pr\"ufer domains via their value groups. Extending work of Puninski and Toffalori, we show that the extended positive cone of the value group of a Pr\"ufer domain has m-dimension if and only if its lattice of pp-$1$-formulae has breadth (equivalently width) and that these dimensions are equal. Further, we show that the existence of these dimensions is equivalent to the lattice of pp-$1$-formulae having m-dimension (and hence to its Ziegler spectrum having Cantor-Bendixson rank) and the non-existence of superdecomposable pure-injective modules. Finally, we give a best possible upper bound for the m-dimension of the pp-$1$-lattice of a Pr\"ufer domain in terms of the m-dimension of the extended positive cone of its value group and show that all ordinals which are not of the form $\lambda+1$ for $\lambda$ a limit ordinal occur as the m-dimension of the pp-$1$-lattice of a B\'ezout domain.
\end{abstract}

\maketitle

\tableofcontents

\noindent
It is often the case that results in model theory of modules over B\'ezout domains can be proved, with sometimes significantly more effort, for Pr\"ufer domains. The first main contribution of this paper is to prove a transfer theorem which allows results in model theory of modules over B\'ezout domains to be translated into results over Pr\"ufer domains.

A B\'ezout domain is an integral domain whose finitely generated ideals are principal and a Pr\"ufer domain is an integral domain such that all its localisations at maximal ideals are valuation domains. B\'ezout domains are the non-Noetherian analogues of principal ideal domains and Pr\"ufer domains are non-Noetherian analogues of Dedekind domains.
Many classically important rings are B\'ezout or Pr\"ufer domains. They include Dedekind domains and hence rings of integers of number fields; the ring of complex entire functions \cite[Thm. 9]{Helmer} and the ring of algebraic integers \cite[Thm. 102]{Kapcomm}; the ring of integer valued polynomials with rational coefficients \cite[VI.1.7]{IntPoly} and the real holomorphy rings of a formally real fields \cite[2.16]{Becker}.

The central objects of study in the model theory of modules are lattices of pp formulae and Ziegler spectra. A pp formulae over a ring $R$ is a conjunction of linear equations over $R$ with some existential quantifiers in front;
see section \ref{SPreI}. Every definable set in a fixed module is the solution set of a boolean combination of pp formulae. For each natural number $n$, the set of pp-$n$-formulae i.e. pp formulae in $n$ free variables, up to equivalence, is a lattice under implication or equivalently inclusions of solution sets. Note that the lattice of pp-$n$-formulae of a ring $R$ is isomorphic to the lattice of finitely presented subfunctors of $\Hom_R(R^n,-)$ in the category of additive functors from the category of finitely presented $R$-modules to abelian groups.  The Ziegler spectrum, $\Zg(R)$, of a ring $R$ is a topological space attached to its module category. The points of $\Zg(R)$ are isomorphism classes of indecomposable pure-injective modules and a basis of compact open sets, $\mcal{O}(\Zg(R))$, is given by the sets
\[\left(\nf{\phi}{\psi}\right):=\{N\in\Zg(R) \st \phi(N)\supsetneq\phi(N)\cap\psi(N)\}\] where $\nf{\phi}{\psi}$ are pairs of pp-$1$-formulae.
This space, together with the lattices of pp formulae, captures most of the model theoretic information about the category of $R$- modules.

It was shown in \cite{bezwidth}, that both the lattice of pp-$1$-formulae, $\pp_R^1$, and the Ziegler spectrum, $\Zg(R)$ of a B\'ezout domain $R$ are encoded in its value group $\Gamma(R)$. The value group of a B\'ezout domain is simply the lattice ordered abelian group $Q^\times/R^\times$ where $Q$ is the field of fractions of $R$ ordered such that $a\leq b$ if and only if $bR\subseteq aR$. Replacing elements of $R$ by finitely generated fractional ideals, \`a la Dedekind, we define the value group $\Gamma(R)$ of a Pr\"ufer domain $R$ to be the group of non-zero fractional ideals of $R$ ordered by reverse ideal inclusion. This allows us to prove the following theorem.

{
    \renewcommand{\thetheorem}{\ref{BeztoPruf}}
    \begin{theorem}
     Let $R,S$ be Pr\"ufer domains. If $\Gamma(R)\cong\Gamma(S)$ then there exists a lattice isomorphism $\lambda:\pp_S^1\rightarrow \pp_R^1$ and a homeomorphism of topological spaces $\rho:\Zg(R)\rightarrow \Zg(S)$ such that $(\lambda\phi/\lambda\psi)=\rho^{-1}(\phi/\psi)$.
    \end{theorem}
}

By the Jaffard-Kaplansky-Ohm theorem, all lattice ordered abelian groups occur as the value group of a B\'ezout domain. Thus, any result about the triple $(\pp_R^1,\Zg(R),\pi_R)$, where \[\pi_R:\pp_R^1\times\pp_R^1\rightarrow \mcal{O}(\Zg(R)), \ \ \ (\phi,\psi)\mapsto \left(\nf{\phi}{\psi}\right),\] which is true for all B\'ezout domains (with value group $G$) is also true for all Pr\"ufer domains (with value group $G$). The strength of this theorem is that many results in model theory of modules, for an arbitrary ring $S$, can be rephrased in terms of the triple $(\pp_R^1,\Zg(R),\pi_R)$.

The lattice of pp-$1$-formulae of a ring is often too coarse an invariant on its own. For instance, very different valuation domains have isomorphic lattices of pp-$1$-formulae: All valuation domains with countable dense value groups have isomorphic lattices of pp-$1$-formulae. On the other hand, for $n\geq 2$, the lattice of pp-$n$-formulae of $R$ is often a much less tame (or more descriptive) object. For instance, for $K$ and $F$ fields, if $\pp_K^3$ is isomorphic to $\pp_F^3$ as a lattice then $K$ is isomorphic to $F$ as a field. This is because the lattice $\pp_K^3$ is isomorphic to the lattice of subspaces of $K^3$; a classical result, \cite{GeoAlg}, recovers $K$ from the lattice of subspaces of $K^3$.

Part of our original motivation for proving this transfer theorem was that it would allow us to transfer results of Toffalori and Puninski from \cite{bezwidth} about the width of pp-1-lattices over B\'ezout domains to Pr\"ufer domains. Whilst preparing to do this we realised that by working more extensively with value groups we could extend those results significantly. The second main contribution of this article is the following theorem.

\begin{theorem*}
Let $R$ be a Pr\"ufer domain. The following are equivalent.
\begin{enumerate}
\item The lattice $\pp^1_R$ of pp-$1$-formulae has breadth (equivalently width).
\item The lattice $\pp^1_R$ of pp-$1$-formulae has m-dimension (equivalently $\Zg(R)$ has Cantor-Bendixson rank).
\item There does not exist a superdecomposable pure-injective $R$-module.
\item The extended positive cone $\Gamma(R)_\infty^+$ of the value group of $R$ has m-dimension.
\end{enumerate}
Moreover, the m-dimension of the extended positive cone of the value group of $R$ is equal to the breadth of $\pp_R^1$.
\end{theorem*}

Breadth and m-dimension are dimensions on modular lattices defined by iteratively collapsing intervals which belong to a chosen class of lattices $\mcal{L}$ closed under sublattices and quotients; total orders in the case of breadth and finite lattices in the case of m-dimension. Roughly speaking, see section \ref{PreII} for a proper definition, they measure the number of times we need to collapse the intervals in $\mcal{L}$ before we reach the trivial lattice.

The m-dimension of the lattice of pp-$n$-formulae is equal to the Krull-Gabriel dimension, in the sense of \cite{Geigle}, of the category $(\mod\text{-}R,\Ab)^{fp}$ of finitely presented functors from $\mod\text{-}R$, the category of finitely presented right $R$-modules, to $\Ab$, the category of abelian group. Krull-Gabriel dimension measures the complexity of the factorisations of homomorphisms in the category of finitely presented modules.

Breadth coexists with another dimension, the ``width'' of a modular lattice. Ziegler showed, \cite[7.1]{Zieglermodules}, that if there is a superdecomposable pure-injective $R$-module then $\pp_R^1$ does not have width and, that, when $R$ is countable, the converse holds. Recall that the interest in the existence of superdecomposable pure-injective modules comes from the fact that, over any ring, every pure-injective module is isomorphic to the pure-injective hull of a direct sum of indecomposable pure-injective modules and a superdecomposable pure-injective module.

Most proofs of the existence of superdecomposable pure-injective modules are proved by showing that $\pp_R^1$ (or one of its analogues) does not have width; a notable exception is \cite{PunconcreteSDstring}. However, it is not known in general if the assumption that $R$ is countable in Ziegler's result can be dropped for the converse to hold. It is known for serial rings \cite[6.4]{SuperdecompSerial} and von Neumann regular rings \cite[1.2,1.3]{Trlifaj2prob}. The implication (3)$\Rightarrow$(1) of our theorem extends it to Pr\"ufer domains.

The proof of the theorem is spread over sections \ref{SDimLgroup}, \ref{SSuperDecPIM} and \ref{SCalcBreadth}. In particular
(3)$\Rightarrow$(4) is \ref{vgroupmdimsuperdec}, (2)$\Leftrightarrow$(4) is \ref{weakmdim} and
(1)$\Leftrightarrow$(4) plus the statement about the value of breadth is \ref{calcbreadth}.
Note that, as we have already mentioned, (1)$\Rightarrow$(3) is true for all rings.
Most parts of the theorem are proved for B\'ezout domains and then transferred to Pr\"ufer domains using our transfer theorem \ref{BeztoPruf}.


That $(1)$ and $(4)$ are equivalent for B\'ezout domains is due to Puninski and Toffalori in \cite[7.1]{bezwidth}. The difficult direction is $(4)$ implies $(1)$.
Their ingenious proof first proves the equivalence in the countable case and then derives the equivalence in the general case by using the downwards L\"owenheim-Skolem theorem.
They first show, \cite[7.4 \& 7.7]{bezwidth}, that if a B\'ezout domain $R$ has a superdecomposable pp-type then $\Gamma(R)^+_\infty$ does not have m-dimension. So, in particular, they show $(3)\Rightarrow (4)$. Hence, if $R$ is countable then, by Ziegler's result, $\pp_R^1$ has width if and only if $\pp_R^1$ has a superdecomposable pp-type. This proves $(4)$ implies $(1)$ when $R$ is countable.

We give a very different proof of $(4)$ implies $(1)$ which allows us to calculate the breadth of $\pp^1_R$. We show, \ref{breadthcollapseloc}, that, when $R$ is a B\'ezout domain which is not a field, collapsing totally ordered intervals in $\pp_R^1$ corresponds to localising at the multiplicatively closed subset generated by the irreducible elements of $R$. At the level of value groups this corresponds to collapsing the intervals of finite length in the extended positive cone, or, equivalently, factoring out the convex $\ell$-subgroup generated by the minimal positive elements, see \ref{dimforgroup1} and \ref{dimintermsofmult}.
The tools to prove (4) implies (1) have applications. They allow us to show, \ref{locsuperdecomp}, that for all B\'ezout domains $R$ there exists a mutiplicatively closed subset $S_\infty$ of $R$ such that no non-trivial interval in $\pp^1_{R_{S_\infty}}$ is totally ordered and all superdecomposable pure-injective $R$-modules are $R_{S_\infty}$-modules.
They also allow us to prove (4) implies (2). However they do not give any information about the exact value of m-dimension.

For any ring $R$, when $\pp_R^1$ has m-dimension, it is equal to the Cantor-Bendixson (CB) rank of $\Zg(R)$. For Pr\"ufer domains, Puninski shows that these dimension are always equal.
 The final contribution of this article is an investigation of the CB rank of $\Zg(R)$ when $R$ is a Pr\"ufer domain. This is the content of section \ref{SCBrank}. We show that if $\alpha$ is the m-dimension of $\Gamma(R)_\infty^+$ then the CB rank of $\Zg(R)$ is bounded below by the ordinal $\alpha$ and above by $\alpha \cdot 2$.
 By reinterpreting a result of Puninski we see that for a valuation domain $V$, the m-dimension of $\pp_V^1$ is equal to $\alpha\cdot 2$ where $\alpha$ is the m-dimension of $\Gamma(V)_\infty^+$.
 Hence our upper bound is best possible. In the opposite extreme case, we consider Pr\"ufer domains of Krull dimension $1$, i.e. all non-zero prime ideals are maximal.
 In this situation, we show that if the m-dimension of $\Gamma(R)_\infty^+$ is $\alpha$ then the m-dimension of $\pp_R^1$ is equal to $\alpha$ if $\alpha$ is a limit ordinal and $\alpha+1$ otherwise.

 That a Pr\"ufer domain has Krull dimension $1$ is a property of its value group. For each ordinal $\alpha$, we construct,
 in section \ref{Sctsfn},
 a lattice ordered abelian group $\Gamma$ of ``Krull dimension $1$'' such that the m-dimension of $\Gamma_\infty^+$ is $\alpha$.
 This allows us to show that all ordinals which are not of the form $\lambda+1$ for $\lambda$ a limit ordinal occur as the CB rank of the Ziegler spectrum of a B\'ezout domain.

In order to construct a lattice ordered abelian group $\Gamma$ of ``Krull dimension $1$'' with $\mdim\Gamma_\infty^+=\alpha$ for each ordinal $\alpha$, we first consider groups of continuous functions $C(X,\Z)$ where $X$ is a boolean space and $\Z$ is equipped with the discrete topology. We show that if $X$ has CB rank $\beta$ then $C(X,\Z)_\infty^+$ has m-dimension $\beta+1$. Thus we miss all limit ordinals. To remedy this, we take a boolean space of CB rank $\alpha$ and let $C^-(X,\Z)$ be the convex lattice ordered subgroup of $C(X,\Z)$ of those $f\in C(X,\Z)$ with $f(x)=0$ for all $x\in X$ of CB rank $\alpha$. Like $C(X,\Z)$ this group has ``Krull dimension $1$'' but $\mdim C^-(X,\Z)=\alpha$ when $X$ has CB rank $\alpha$ as required.

It is not hard to see, \ref{mdim1}, that the CB rank of the Ziegler spectrum of a Pr\"ufer domain is never equal to 1. We conjecture the following.

\begin{conjecture*}
Let $\lambda$ be a limit ordinal. There does not exist a Pr\"ufer domain whose Ziegler spectrum has CB rank $\lambda+1$.
\end{conjecture*}

This paper has two preliminary sections \ref{SPreI} and \ref{PreII}. The first one contains a brief introduction to the notions we need from model theory of modules, highlighting what happens for arithmetical rings;
 a discussion of  weakly prime ideals of Pr\"ufer domains;  and a summary of the connections between ideal theory of Pr\"ufer domains and multiplicatively closed subsets of B\'ezout domains and their value groups as lattice ordered abelian groups.
 The second preliminary section, which is not needed for the transfer theorem, is about dimensions on modular lattices, particularly pp-lattices and superdecomposable pp-types.

\medskip

%

%
%
%

\section{Preliminaries I}\label{SPreI}
\noindent
We will write $\Ord$ for the class of ordinals. For $R$ a ring, $\Mod\text{-}R$ denotes the category of (right) $R$-modules.

A commutative ring $R$ is \textbf{B\'ezout} if all finitely generated ideals of $R$ are principal and \textbf{arithmetical} if $R_\mfrak{m}$ is a valuation ring for all maximal ideals $\mfrak{m}$. Note that B\'ezout rings are arithmetical. An integral domain which is arithmetical is called a \textbf{Pr\"ufer domain}.

We list some useful properties which characterise arithmetical rings. We will see some more later in this section.
\begin{theorem}\label{arithdef}Let $R$ be a commutative ring. The following are equivalent.
\begin{enumerate}
\item $R$ is arithmetical.
\item The lattice of ideals of $R$ is distributive.
\item For all ideals $A,B$ and finitely generated ideal $C$\[(A+B:C)=(A:C)+(B:C)\]\label{idealquotientarith}
\item For all $a,b\in R$, there exists $r,s,\alpha\in R$ such that $a\alpha=br$ and $b(1-\alpha)=as$.\label{arithdefalpha}
\end{enumerate}
\end{theorem}
\begin{proof}
See \cite[Thm 1 \& 2]{JensenArith} for the equivalence of (1),(2) and (3). That (2) and (4) are equivalent follows from \cite[Thm 1.6]{Stephenson}.
\end{proof}

\begin{remark}\label{genint}
Let $R$ be an arithmetical ring. Let $a,b\in R$ and $\alpha,r,s\in R$ be such that $a\alpha=br$ and $b(1-\alpha)=as$. If $c\in aR\cap bR$ then $c=a\lambda=b\mu$ for some $\lambda,\mu\in R$ and hence \[c=c\alpha+c(1-\alpha)=br\lambda+as\mu.\] In particular, $aR\cap bR$ is generated by $\{br, as\}=\{a\alpha,b(\alpha-1)\}$.
\end{remark}

Since the lattice of ideals of an arithmetical ring is distributive, it follows from the above remark that the intersection of two finitely generated ideals is finitely generated. Therefore the set of finitely generated ideals of an arithmetical ring is a sublattice of its lattice of ideals.


\subsection*{Model Theory of Modules}\label{SPreIModTh}
We now give a summary of the notions from model theory of modules that will
be used in this paper, highlighting special features for arithmetical rings. For a more detailed introduction to general model theory of modules the reader is referred to \cite{PrestBluebook} and  \cite{PSL}.

A (right) pp-$n$-formula $\phi$ is a formula of the form \[\exists \overline{y} \ \overline{y}A=\overline{x}B\]
where $\overline{x}$ is an $n$-tuple of variables, $\overline{y}$ is a tuple of variables and $A, B$ are appropriately sized matrices with entries from $R$. We write $\phi(M)$ for the solution set of $\phi$ in $M$. Note that $\phi(M)$ is a subgroup of $M^n$ with the addition inherited from $M$.

We identify two pp formulae if they define the same subgroup in all $R$-modules. Once we have made this identification, the set of (right) pp-$n$-formulae becomes a lattice under inclusion of solution sets, i.e. $\psi\leq \phi$ if and only if $\psi(M)\subseteq \phi(M)$ for all $M\in\Mod\text{-}R$. We denote this lattice by $\pp_R^n$ and the left module version by ${_R}\pp^n$. In \textit{this} lattice, we write $\phi\wedge \psi$ for the meet and $\phi+\psi$ for the join of $\phi,\psi\in \pp_R^n$. Note that for all $M\in \Mod\text{-}R$, $\phi\wedge\psi(M)=\phi(M)\cap\psi(M)$ and $(\phi+\psi)(M)=\phi(M)+\psi(M)$. 

For each $n\in\N$, Prest introduced a lattice anti-isomorphism $D:\pp_R^n\rightarrow {_R}\pp^n$ (see \cite[section 1.3.1]{PSL} and \cite[8.21]{PrestBluebook}). As is standard, we denote its inverse also by $D$. We won't give the definition of $D$ because everything we need about it can be derived from the fact that for all $a\in R$, $D(xa=0)$ is $a|x$ and $D(a|x)$ is $ax=0$ where $a|x$ denotes the pp formula $\exists y \ x=ya$. Since all the rings we are interested in here are commutative we identify $\pp_R^n$ and ${}_R\pp^n$, so $D$ becomes an anti-isomorphism $D:\pp_R^1\rightarrow \pp_R^1$.

If $M$ is an $R$-module and $m\in M$ then the pp-type, $\pp_M(m)$ of $m$ (in $M$) is the collection of pp-$1$-formulae $\phi$ such that $m\in\phi(M)$. Note that $\pp_M(m)$ is a filter in $\pp^1_R$ and all filters in $\pp_R^1$ occur as the pp-types of elements of $R$-modules. If $M$ is finitely presented and $m\in M$ then $\pp_M(m)$ is generated as a filter by a single pp formula. Moreover, if $\phi$ is a pp formula then there exists a finitely presented module $M$ and $m\in M$ such that $\phi$ generates $\pp_M(m)$. The pair $(M,m)$ is called a \textbf{free realisation} of $\phi$.
\begin{theorem}\cite[Theorem 3]{Warfieldarith}\label{Warfield}
A commutative ring $R$ is arithmetical if and only if every finitely presented $R$-module is a direct summand of a (finite) direct sum of modules of the form $R/rR$ where $r\in R$.
\end{theorem}

Using free realisations and Prest's duality one can derive the following special forms for pp-$1$-formulae over arithmetical rings.

\begin{proposition}\cite[4.5]{ringdesc}\label{arithspecpp}
Let $R$ be an arithmetical ring. Every pp-$1$-formula over $R$ is equivalent to a sum of pp-$1$-formula of the form $\exists y\, (ya=x\wedge yb=0)$ and a conjunction of pp-$1$-formula of the form $\exists y \, yc=xd$.
\end{proposition}

We will use the forward direction of the next characterisation of arithmetical rings frequently.

\begin{theorem}\cite[3.1]{EH}
A commutative ring is arithmetical if and only if $\pp_R^1$ is distributive.
\end{theorem}

An embedding $f:A\rightarrow B$ is \textbf{pure} if for all $\phi\in\pp_R^1$, $f(\phi(A))=\phi(B)\cap f(A)$. An $R$-module $M$ is \textbf{pure-injective} if it is injective over all pure-embeddings i.e. for all $f:A\rightarrow B$ pure and $g:A\rightarrow M$, there exists $h:B\rightarrow M$ such that $h\circ f=g$.

For any pp-type $p$ there exists a pure-injective module $M$ and an element $m\in M$ with $p=\pp_M(m)$ such that for any pure-injective pure submodule $M'$ of $M$, if $m\in M'$ then $M=M'$. The module $M$ is determined up to isomorphism by $p$. We call it the \textbf{pure-injective hull} of $p$ and write $N(p)$. See \cite[3.6]{Zieglermodules}, \cite[Ch 4]{PrestBluebook} or \cite[\S 4.3.5]{PSL}.

A pp-type $p$ is said to be \textbf{indecomposable} if $N(p)$ is indecomposable. Note that all indecomposable pure-injective modules are of the form $N(p)$ for some indecomposable pp-$1$-type $p$. By \cite[4.4]{Zieglermodules}, pp-type $p$ is indecomposable if and only if for all $\phi_1,\phi_2\notin p$, there exists $\sigma\in p$ such that $\phi_1\wedge \sigma+\phi_2\wedge \sigma\notin p$. When $\pp_R^1$ is distributive, so in particular for arithmetical rings, this condition simplifies to give that a pp-$1$-type is indecomposable if and only if $p$ is prime as a filter in $\pp_R^1$ i.e. for all $\phi_1,\phi_2\notin p$, $\phi_1+\phi_2\notin p$.

We say that an $R$-module $M$ is \textbf{pp-uniserial} if for all $\phi,\psi\in \pp^1_R$, $\phi(M)$ and $\psi(M)$ are comparable. An $R$-module is \textbf{endo-uniserial} if it is uniserial as a module over its endomorphism rings. The next result, which is a simplified version of a result of Puninski, gives two more characterisations of arithmetical rings amongst commutative rings.

\begin{theorem}\label{endouni}\cite[3.3]{KGdimSer}
For any ring $R$, the following are equivalent.
\begin{enumerate}
\item The lattice $\pp_R^1$ is distributive.
\item Every indecomposable pure-injective $R$-module is pp-uniserial.
\item Every indecomposable pure-injective $R$-module is endo-uniserial.
\end{enumerate}
\end{theorem}

The (right) \textbf{Ziegler spectrum} of a ring $R$, denoted $\zg(R)$, is a topological space whose points are isomorphism classes of indecomposable pure-injective (right) $R$-modules and which has a basis of open sets given by
\[(\phi / \psi)=\{M\in\pinj_R \st \phi(M)\supsetneq\psi(M)\and\phi(M)\},\]
where $\varphi,\psi$ range over (right) pp-$1$-formulae. The sets $(\phi / \psi)$ are compact, in particular, $\Zg(R)$ is compact.
For countable rings \cite[4.7]{Herzogduality} and arithmetical rings \cite[2.9]{SerialSob}\footnote{In that paper, arithmetical rings are referred to a Pr\"ufer rings.}, $\Zg(R)$ is known to be a sober space.

\smallskip

Localisation of rings at multiplicatively closed sets will be important in this paper.
Let $R$ be a commutative ring and $S$ a multiplicatively closed subset of $R$. Let $\iota_S:R\rightarrow R_S$ be the localisation map. Restriction of scalars along $\iota$ induces a full embedding of $\Mod\text{-}R_S$ into $\Mod\text{-}R$ and applying $\iota$ to the matrices defining $\phi\in \pp_R^1$ gives a surjective lattice homomorphism $\pi_S:\pp_R^1\rightarrow \pp_{R_S}^1$. One easily sees that if $M\in \Mod\text{-}R_{S}$ and $\phi\in \pp_R^1$ then $\phi(M_R)=\pi_S(\phi)(M_{R_S})$. In particular, if $M$ is an $R_S$-module and $m\in M$ then $\phi\in \pp_{M_R}(m)$ if and only if $\pi_S(\phi)\in \pp_{M_{R_S}}(m)$.

Since $\iota_S$ is an epimorphism, by \cite[5.5.3]{PSL}, restriction of scalars along $\iota_S$ preserves both indecomposability and pure-injectivity. Moreover, the map from $\Zg(R_S)$ to $\Zg(R)$ induced by restriction of scalars is a homeomorphism onto its image, which is a closed set. As a closed subset of $\Zg(R)$, we may identify $\Zg(R_S)$ with the complement of the union of the open sets of the form $\left(\nf{x=x}{s|x}\right)$ and $\left(\nf{xs=0}{x=0}\right)$ where $s\in S$. We will often identify $\Zg(R_S)$ with this closed set.\looseness=-1

Part (i) of the next lemma follows from the fact, \cite[4.3.43]{PSL}, that the endomorphism rings of indecomposable pure-injective modules are local. A proof of (ii) can be found in \cite[6.4]{LornaSob}.\looseness=-1

\begin{lemma}\label{Zglocglob}
Let $R$ be a commutative ring $R$.
\begin{enumerate}[(i)]
\item Let $N$ be an indecomposable pure-injective $R$-module. The set $\Att N$ of elements of $R$ whose action on $N$ by multiplication is not bijective is a prime ideal of $R$.
\item For each prime ideal $\mfrak{p}\lhd R$, we can identify $\Zg(R_\mfrak{p})$ with the closed subset of $\Zg(R)$ consisting of those $N\in\Zg(R)$ with $\Att N\subseteq \mfrak{p}$. Moreover, $\Zg(R)$ is equal to the union of $\Zg(R_\mfrak{p})$ where $\mfrak{p}\lhd R$ ranges over maximal ideals of $R$.
\end{enumerate}
\end{lemma}

%

\subsection*{Weakly prime ideals}
In \cite{EntBez}, an ideal $I$ of a B\'ezout domain $R$ was defined to be weakly prime if for all $a,b\in R$, $a,b\notin I$ implies the least common multiple of $a$ and $b$ (i.e. a generator of $aR\cap bR$) is not in $I$. We modify this definition so that it is more appropriate for arithmetical rings and our purposes.

\begin{definition}
An ideal $I\lhd R$ is \textbf{weakly prime} if $I$ is proper and for all finitely generated ideals $A,B\lhd R$, if $A\cap B\subseteq I$ then $A\subseteq I$ or $B\subseteq I$.
\end{definition}

Over arithmetical rings, weakly prime ideals coincide with various other classes of ideals.

\begin{proposition}\label{eqwkprime}
Let $R$ be an arithmetical ring and $I\lhd R$ be a proper ideal. The following are equivalent:
\begin{enumerate}
\item $I$ is weakly prime.
\item $I$ is irreducible, i.e. for all $A,B\lhd R$, $A\cap B=I$ implies $A=I$ or $B=I$.
\item $I$ is primal i.e. $\bigcup_{r\notin I}(I:r)$ is an ideal.
\end{enumerate}
\end{proposition}
\begin{proof}
The implications $(1)\Rightarrow(2)$ and $(2)\Rightarrow(3)$ hold for all commutative rings. The proof of $(1)\Rightarrow (2)$ is straightforward. See \cite[Thm 1]{Fuchsprimal} for $(2)\Rightarrow (3)$.
According to \cite[1.8]{primal2}, $(3)\Rightarrow (2)$ characterises arithmetical rings. The implication $(2)\Rightarrow (1)$ follows easily from the fact that the ideal lattice of an arithmetical ring is distributive. 
\end{proof}

For $I\lhd R$ a proper primal ideal, denote the ideal $\bigcup_{r\notin I}(I:r)$ by $I^\#$. Note that $I^\#$ is a prime ideal.
One can easily check directly\footnote{For irreducibility, note that if $A\cap B=(I:r)$ then $(Ar+I)\cap (Br+I)=I$. For the second claim, note that if $r,s\notin I$ then $(rR+I)\cap (sR+I)\subsetneq I$. If $t=r\delta+\lambda\in (rR+I)\cap (sR+I)$ with $\lambda\in I$ and $t\notin I$ then $(I:t)=((I:r):\delta)\supseteq (I:s)$} that, for any commutative ring $R$, if $I\lhd R$ is irreducible and $r\notin I$ then $(I:r)$ is irreducible and $(I:r)^\#=I^\#$.
%
%
\begin{lemma}
Let $I\lhd R$ be primal. Then $I^\#$ is the union of all ideals $(I:K)$ where $K\lhd R$ is a finitely generated ideal with $K\nsubseteq I$.
\end{lemma}
\begin{proof}
Suppose $K=a_1R+\ldots+a_nR$. Then $(I:K)=\bigcap_{i=1}^n (I:a_i)$. If $K\nsubseteq I$ then $a_i\notin I$ for some $1\leq i\leq n$. So $(I:K)\subseteq (I:a_i)\subseteq I^\#$.
\end{proof}

We will frequently use the following standard result from commutative ring theory which allows us to check certain ideal equations and inclusions locally.

\begin{lemma}[]\label{locidealeq}
Let $R$ be a commutative ring and $I,J\lhd R$ be ideals. Then $I\subseteq J$ if and only if $IR_{\mfrak{p}}\subseteq JR_{\mfrak{p}}$ for all prime ideals $\mfrak{p}\lhd R$. Moreover, if $\mfrak{p}\lhd R$ is a prime ideal then
\begin{enumerate}

\item $(I\cap J)R_{\mfrak{p}}=IR_{\mfrak{p}}\cap JR_{\mfrak{p}}$,
\item $(I+J)R_{\mfrak{p}}=IR_{\mfrak{p}}+JR_{\mfrak{p}}$,
\item $(I\cdot J)R_{\mfrak{p}}=IR_{\mfrak{p}}\cdot JR_{\mfrak{p}}$ and,
\item when $J$ is finitely generated, $(I:J)R_{\mfrak{p}}=(IR_{\mfrak{p}}:JR_{\mfrak{p}})$.
\end{enumerate}
\end{lemma}

Let $\mfrak{p}\lhd R$ be a prime ideal and let $\iota:R\rightarrow R_\mfrak{p}$ be the localisation map. For $X\subseteq R_\mfrak{p}$, we write $X\cap R$ for the set $\iota^{-1}(X)$.

\begin{lemma}\label{formwkprime}
Let $R$ be an arithmetical ring and $I\lhd R$ a proper ideal. Then $I$ is weakly prime if and only if $I=IR_{\mfrak{p}}\cap R$ for some prime ideal $\mfrak{p}$. Moreover, for $I\lhd R$ weakly prime, $I=IR_{\mfrak{p}}\cap R$ for any prime ideal $\mfrak{p}\supseteq I^\#$.
\end{lemma}
\begin{proof}
The reverse direction follows easily from \ref{locidealeq} and the fact that all ideals of a valuation ring are weakly prime. For the forward direction, let $\mfrak{p}\supseteq I^\#$.
It is always true that $I\subseteq IR_{\mfrak{p}}\cap R$. Suppose $a\in IR_{\mfrak{p}}\cap R$. There exists $s\notin \mfrak{p}$ such that $as\in I$ for some $s\notin \mfrak{p}$. So $s\in (I:a)$. If $a\notin I$ then $s\in (I:a)\subseteq \mfrak{p}$. Therefore, $s\notin\mfrak{p}$ implies $a\in I$. So we have shown that $I=IR_{\mfrak{p}}\cap R$.
\end{proof}

%

As a consequence, for each prime ideal  $\mfrak{p}\lhd R$
of an arithmetical ring $R$, the set of weakly prime ideals $I\lhd R$ such that $I^\#\subseteq \mfrak{p}$ is totally ordered by inclusion.


\smallskip
For any commutative ring $R$, ideal $I\lhd R$, prime ideal $\mfrak{p}\lhd R$ and finitely generated ideal $K$,
\[((I:K)R_\mfrak{p})\cap R=(IR_\mfrak{p}:KR_\mfrak{p})\cap R=(IR_\mfrak{p}\cap R:K).\]
Combined with \ref{formwkprime} this gives the the first item in the following remark. The second item follows from the definition of $J^\#$.

\begin{remark}\label{hashandloc}
Let $R$ be an arithmetical ring.
\begin{enumerate}[(1)]
\item 
Let $K\lhd R$ be a finitely generated ideal and $I\lhd R$ a weakly prime ideal. If $\mfrak{p}\lhd R$ is a prime ideal with $\mfrak{p}\supseteq I^\#$ then $(IR_\mfrak{p}:KR_\mfrak{p})\cap R=(I:K)$. 
\item Let $\mfrak{p}\lhd R$ a prime ideal. For any proper ideal $J\lhd R_\mfrak{p}$, $(J\cap R)^\#=J^\#\cap R\subseteq \mfrak{p}$.
\end{enumerate}
\end{remark}

When $R$ is a valuation domain, the ideal quotient operation has an inverse: If $I\lhd R$ an ideal, $s\in R$ and $r\notin I$, then $(I:r)\cdot r=I$ and $(I\cdot s:s)=I$. This does not generalise naively to Pr\"ufer domains, even for weakly prime ideals. Moreover, for $I\lhd R$ weakly prime and $r\in R$, the ideal $Ir$ may not be weakly prime.  An operation inverse to the ideal quotient on weakly prime ideals was introduced in \cite{bezwidth} for B\'ezout domains. We give a  version of this definition which works for Pr\"ufer domains.

For $K\lhd R$ a non-zero finitely generated ideal and $J\lhd R$ a weakly prime ideal, define
\[J_K:=\{a\in R \st (aR:K)\subseteq J\}.\]
Note that, by \ref{arithdef}, when $R$ is arithmetical, $J_K$ is an ideal of $R$ because for all $a,b\in R$, $((a+b)R:K)\subseteq (aR+bR:K)=(aR:K)+(bR:K)$.

\begin{lemma}\label{InverseIdealQuotient} Let $R$ be a Pr\"ufer domain, $J\lhd R$ a weakly prime ideal and $K\lhd R$ a non-zero finitely generated ideal.
\begin{enumerate}
\item For all prime ideals $\mfrak{p}$ with $\mfrak{p}\supseteq J^\#$, $J_K=JKR_\mfrak{p}\cap R$. Hence, $J_K$ is weakly prime.
\item $(J_K)^\#=J^\#$.
\item $(J_K:K)=J$.
\item If $K\nsubseteq J$ then $(J:K)_K=J$.
\end{enumerate}
\end{lemma}
\begin{proof}(1) Suppose $a\in J_K$. Then $(aR:K)\subseteq J$. So $(aR_\mfrak{p}:KR_\mfrak{p})\subseteq JR_\mfrak{p}$. Since $JR_\mfrak{p}$ is a proper ideal, $KR_\mfrak{p}\nsubseteq aR_\mfrak{p}$. Therefore, since $R_\mfrak{p}$ is a valuation domain,
\[aR_\mfrak{p}=(aR_\mfrak{p}:KR_\mfrak{p})\cdot KR_\mfrak{p}\subseteq JKR_\mfrak{p}.\] Hence $a\in JKR_\mfrak{p}\cap R$. Now suppose that $a\in JKR_\mfrak{p}\cap R$. Since $R_\mfrak{p}$ is a valuation domain, $(aR_\mfrak{p}:KR_\mfrak{p})\subseteq JR_\mfrak{p}$. Therefore
\[(aR:K)\subseteq (aR:K)R_\mfrak{p}\cap R=(aR_\mfrak{p}:KR_\mfrak{p})\cap R\subseteq JR_\mfrak{p}\cap R=J.\] So $a\in J_K$.

\smallskip

\noindent
(2) This follows from (1) and \ref{hashandloc}(1).

\smallskip

\noindent
(3) Let $\mfrak{p}:=J^\#$. Then
\[(J_K:K)R_\mfrak{p}=(J_KR_\mfrak{p}:KR_\mfrak{p})=(JKR_\mfrak{p}:KR_\mfrak{p})=JR_\mfrak{p}.\] The second equality is a consequence of (1) and the third holds because $R_\mfrak{p}$ is a valuation domain. By (1), $J_K$ is weakly prime and since $J$ is proper, $K\nsubseteq J_K$. So $(J_K:K)$ is weakly prime. By (2) and (3), $(J_K:K)^\#=J^\#$. Therefore $(J_K:K)=(J_K:K)R_\mfrak{p}\cap R$. Thus, since $JR_\mfrak{p}\cap R=J$, the displayed equalities imply $(J_K:K)=J$. A similar argument, which we leave to the reader, proves (4).
%
\end{proof}

Note that as a consequence of \ref{InverseIdealQuotient}(1),
for all non-zero finitely generated ideals $K\lhd R$ and weakly prime ideals $J\lhd R$ of a Pr\"ufer domain $R$, there exists $a\in R\backslash \{0\}$ such that $J_K=J_{aR}$; take $a\in R$ with $aR_{J^\#}=KR_{J^\#}$.

\subsection*{Pr\"ufer domains and their value groups}\label{valuegroups}

Let $R$ be an integral domain with field of fractions $Q$. Recall that a \textbf{fractional ideal} of $R$ is a submodule $J$ of $Q$ such that there exists $r\in R\backslash\{0\}$ with $Jr\subseteq R$. The set of fractional ideals of a ring form a monoid, $\mathfrak{F}(R)$, under ideal multiplication with identity given by $R$. A fractional ideal $J$ is said to be \textbf{invertible} if it is invertible in $\mathfrak{F}(R)$. Recall that all invertible ideals are finitely generated. Let $\Gamma(R)$ denote the group of invertible fractional ideals of $R$ equipped with the order $A\leq B$ if and only if $B\subseteq A$.
We will generally denote finitely generated (fractional) ideals by the letters $A,B,C,D,K$ whereas we use $I,J$ for arbitrary ideals.

Recall, \cite[ChIII 1.1]{MONND}, that $R$ is a Pr\"ufer domain if and only if all its non-zero finitely generated ideals are invertible. We call $\Gamma(R)$ the \textbf{value group} of $R$. Note that for B\'ezout domains, and in particular valuation domains, $\Gamma(R)$ is isomorphic as an ordered group to the usual notion of value group, i.e., to the group $Q^\times/U(R)$, where $U(R)$ denotes the group of units of $R$, ordered by $aU(R)\leq bU(R)$ if and only if $a^{-1}b\in R$.

As already mentioned, \ref{genint}, the intersection of two finitely generated ideals of a Pr\"ufer domain is finitely generated. From this, it easily follows that the intersection of two finitely generated fractional ideals is finitely generated. Thus, when $R$ is a Pr\"ufer domain, $\Gamma(R)$ is a lattice and so, since for all $A,B,K\in \Gamma(R)$, $A\leq B$ implies $AK\leq BK$, $\Gamma(R)$ is a lattice ordered abelian group ($\ell$-group).

\medskip

\noindent
\textbf{Warning:} The supremum of $A,B\in\Gamma(R)$ is the ideal $A\cap B$ and the infimum is the ideal $A+B$.

\medskip

We will follow the convention in valuation theory and write value groups additively. For $\Gamma$ an $\ell$-group, we will write $\Gamma^+$ for the positive cone of $\Gamma$. So $\Gamma(R)^+$ consists of $A\in \Gamma(R)$ such that $A\geq R$, equivalently $A\subseteq R$, i.e. $A$ is a non-zero finitely generated ideal of $R$. We write $\Gamma_\infty$ for the set $\Gamma\cup\{\infty\}$ where $\infty\geq a$ for all $a\in \Gamma$ and $\Gamma^+_\infty$ for $\Gamma^+\cup\{\infty\}$.

There is a bijective correspondence between the ideals of $R$ and the filters of $\Gamma(R)^+_{\infty}$ given by
\[I\lhd R\mapsto \val(I):=\{A\in\Gamma(R)^+ \st A\subseteq I\}\cup\{\infty\}.\]
The inverse is given by sending a filter $\mcal{F}$ of $\Gamma(R)_\infty^+$ to $\sum_{A\in\mcal{F}\backslash\{\infty\}} A$.

In this paper, a prime filter of a lattice $L$ is a \textit{proper} filter $\mcal{F}$  of $L$ such that $a\vee b\in \mcal{F}$ implies $a\in \mcal{F}$ or $b\in\mcal{F}$. We call a \textit{proper} filter $\mcal{F}$ of the extended positive cone $\Gamma_\infty^+$ of an $\ell$-group $\Gamma$ a \textbf{multiplication prime filter} if $a+b\in \mcal{F}$ implies $a\in \mcal{F}$ or $b\in \mcal{F}$. Note that since $a+b\geq a\vee b$ for all $a,b\in \Gamma^+_\infty$, if $\mcal{F}$ is a multiplication prime filter then $\mcal{F}$ is a prime filter.

\begin{remark}
Let $R$ be a Pr\"ufer domain. An ideal $I\lhd R$ is weakly prime (prime) if and only if $\val(I)$ is a prime (multiplication prime) filter  of $\Gamma(R)^+_\infty$.
\end{remark}

For $\mcal{F}$ a prime filter of $\Gamma_\infty^+$, define
\[\mcal{F}^\#:=\{A\in \Gamma^+ \st A+K\in \mcal{F} \text{ for some }K\notin \mcal{F}\}.\] Then, for $I\lhd R$ a weakly prime ideal,  $v(I^\#)=v(I)^\#$.
%

\begin{theorem}[Jaffard-Kaplansky-Ohm Theorem]\cite[5.3 Ch III]{MONND}
Let $\Gamma$ be an $\ell$-group. There exists a B\'ezout domain $R$ such that $\Gamma(R)$ is isomorphic to $\Gamma$.
\end{theorem}

An \textbf{$\ell$-subgroup} of an $\ell$-group $\Gamma$ is a subgroup $S$ of $\Gamma$ such that for all $a,b\in S$, $a\vee b\in S$ (equivalently for all $a,b\in S$, $a\wedge b\in S$). It will often be advantageous to work with the
positive cone of an $\ell$-group rather than the $\ell$-group itself.
There is a bijective correspondence between the convex $\ell$-subgroups of $\Gamma$ and the convex submonoids of $\Gamma^+$ given by sending a convex $\ell$-subgroup $C\subseteq \Gamma$ to $C^+:=C\cap\Gamma^+$ and a convex submonoid $B$ of $\Gamma^+$ to the convex $\ell$-subgroup of elements $a\in \Gamma$ such that $-a\wedge 0, a\vee 0\in B$, or, in this situation equivalently the subgroup generated by $B$, see \cite[7.9]{Darnel}.
The quotient $\Gamma/C$ of an $\ell$-group by a convex $\ell$-subgroup $C$ is an $\ell$-group and the inclusion of $\Gamma^+$ in $\Gamma$ induces an isomorphism from $\Gamma^+/C^+$ to $(\Gamma/C)^+$.

There is a bijective correspondence between the convex $\ell$-subgroups of $\Gamma(R)$ and the saturated multiplicatively closed subsets of a B\'ezout domain $R$ given by sending a convex $\ell$-subgroup $C$ to the set of $a\in R$ with $aR\in C$, see \cite[11.3]{AndFeil}. The inverse map is given by sending a saturated multiplicatively closed subset $S$ to the subgroup generated by $\{sR\st s\in S\}$, or equivalently sending $S$ to the kernel of the surjective homomorphism of $\ell$-groups from $\Gamma(R)$ to $\Gamma(R_S)$ defined by sending $aR\in \Gamma(R)$ to $aR_S\in\Gamma(R_S)$.

\section{A Transfer Theorem}

\noindent
In this section we prove Theorem \ref{BeztoPruf}, which gives a method of transferring results about the model theory of modules of B\'ezout domains to results about the model theory of modules of Pr\"ufer domains via their value groups.


For $R$ a commutative ring  and $A,B\lhd R$ finitely generated ideals, let $xA=0$ denote the pp formula $\bigwedge_{i=1}^nxa_i=0$ where $A=\sum_{i=1}^na_iR$ and let $B|x$ denote the pp formula $\sum_{i=1}^nb_i|x$ where $B=\sum_{i=1}^nb_iR$. Note that, as elements of $\pp_R^1$, these formulae do not depend on the choice of generators of $A$ and $B$.
%

Since Prest's duality $D:\pp_R^1\rightarrow \pp_R^1$ is an anti-isomorphism of lattices, the fact that $D(a|x)$ is $xa=0$ and $D(xa=0)$ is $a|x$ for all $a\in R$
implies that for all finitely generated ideals $A\lhd R$, $D(A|x)$ is $Ax=0$ and $D(xA=0)$ is $A|x$.

\begin{remark}\label{commringsumintpp}
Let $R$ be a commutative ring and $A,B\lhd R$ finitely generated ideals. The following hold:
\begin{enumerate}
\item  $A\subseteq B$ if and only if $A|x\leq B|x$.
\item $A\subseteq B$ if and only if  $xB=0\leq xA=0$.
\item $A|x+B|x$ is equivalent to $(A+B)|x$.
\item $xA=0\wedge xB=0$ is equivalent to $x(A+B)=0$.
\end{enumerate}
\end{remark}

%

Let $\Ideals_{fg}(R)$ denote the set of finitely generated ideals of $R$ ordered by inclusion. The above remark shows that that the map \[\lambda:\Ideals_{fg}(R)\rightarrow \pp_R^1 \ \ \ A\in\Ideals_{fg}(R) \mapsto A|x\in\pp_R^1 \] is a join semi-lattice embedding and the map
\[D\lambda:\Ideals_{fg}(R)^{op}\rightarrow \pp_R^1 \ \ \ A\in\Ideals_{fg}(R)^{op} \mapsto xA=0\in\pp_R^1 \] is a meet semi-lattice embedding.

The set of finitely generated ideals of a commutative ring is a sublattice of the lattice of ideals if and only if the intersection of any two finitely generated ideals is finitely generated. Although arithmetical rings are not always coherent, they do have the property, see \ref{genint}, that the intersection of two finitely generated ideals is always finitely generated.

\begin{lemma}\label{idealslattemb}
Let $R$ be a commutative ring such that the intersection of two finitely generated ideals is finitely generated.

\begin{enumerate}
\item The join semi-lattice embedding $\lambda:\Ideals_{fg}(R)\rightarrow \pp_R^1$ is a lattice embedding if and only if $R$ is an arithmetical ring.
\item The meet semi-lattice embedding $D\lambda:\Ideals_{fg}(R)^{op}\rightarrow \pp_R^1$ is a lattice embedding if and only if $R$ is an arithmetical ring.
\end{enumerate}

\end{lemma}
\begin{proof}
The second statement follows from the first.

\noindent
($\Leftarrow$): Suppose that $R$ is an arithmetical ring. Let $a,b\in R$. By \ref{arithdef}(\ref{arithdefalpha}), there exists $\alpha,r,s\in R$ such that $a\alpha=br$ and $b(1-\alpha)=as$. Then, \ref{genint} $aR\cap bR=a\alpha R+b(1-\alpha)R$. Clearly, $a\alpha|x, b(1-\alpha)|x\leq a|x\land b|x$. Suppose $M\in\Mod\text{-}R$ and $m\in (a|x\land b|x)(M)$. Let $m_1,m_2\in M$ be such that $m=m_1a=m_2b$. Then $m=m\alpha+m(1-\alpha)=m_1a\alpha+m_2b(1-\alpha)$. So $m\in a\alpha|x+ b(1-\alpha)|x$. Thus $a|x\land b|x$ is equivalent to $aR\cap bR| x$. The result now follows for arbitrary finitely generated ideals because both $\Ideals_{fg}(R)$ and $\pp_R^1$ are distributive.

\noindent
($\Rightarrow$): For all $a,b\in R$, $a|x\wedge b|x$ is equivalent to $(aR\cap bR)|x$. Using Prest's duality, this implies $xa=0+xb=0$ is equivalent to $x(aR\cap bR)=0$. The element $1+aR\cap bR\in R/aR\cap bR$ satisfies $x(aR\cap bR)=0$. Thus there exists $\alpha\in  R$ such that $\alpha a\in aR\cap bR$ and $(1-\alpha)b\in aR\cap bR$. By \ref{arithdef}(\ref{arithdefalpha}), $R$ is an arithmetical ring.
\end{proof}

\begin{proposition}\label{prespecialform}
Let $R$ be an arithmetical ring. Every pp-$1$-formula over $R$ is equivalent to a sum of formulae of the form $xa_1=0\wedge xa_2=0\wedge b|x$.
\end{proposition}
\begin{proof}
We know, \ref{arithspecpp}, that every pp-$1$-formula over $R$ is equivalent to a sum of formulae of the form $\exists y \ yc=x\wedge yd=0$. Take $c,d\in R$ and let $\alpha,r,s\in R$ be such that $c\alpha=dr$ and $d(1-\alpha)=cs$. We will show that $\exists y \ yc=x\wedge yd=0$ is equivalent to $x\alpha=0\wedge xs=0\wedge c|x$.

Suppose $M\in\Mod\text{-}R$ and $m\in M$ is such that $mc\alpha=mcs=0$. Then $mc=(m-m\alpha)c$ and $(m-m\alpha)d=mcs=0$. So $mc$ satisfies $\exists y \ yc=x\wedge yd=0$.
Suppose $M\in\Mod\text{-}R$ and $m\in M$ is such that $md=0$. Then $mc\alpha=mdr=0$ and $mcs=md(1-\alpha)=0$. So $mc$ satisfies $x\alpha=0\wedge xs=0\wedge c|x$.
\end{proof}

The above proposition implies that, for $R$ arithmetical, as in the case of B\'ezout rings, $\pp_R^1$ is generated by formulae of the form $xa=0$ and $b|x$. The following corollary may look like a minor improvement on the special forms given in \ref{arithspecpp} but it is rather useful.

\begin{cor}\label{Specialform}
Let $R$ be an arithmetical ring. Every pp-$1$-formula over $R$ is equivalent to a sum of formulae of the form
\[C|x\wedge xD=0\] where $C,D\lhd R$ are finitely generated ideals and equivalent to a conjunction of formulae of the form
\[xA=0+B|x\] where $A,B\lhd R$ are finitely generated ideals.
\end{cor}
\begin{proof}
The first follows from \ref{prespecialform} since, by \ref{commringsumintpp}, $xa_1=0\wedge xa_2=0$ is equivalent to $x(a_1R+a_2R)=0$. The second follows from the first using Prest's duality.
\end{proof}


\noindent
Statements similar to the following appear in \cite[3.2]{bezwidth} for B\'ezout domains and in \cite[12.1, 12.4]{Serialrings} for valuation rings. However, our statement does not directly follow from either of those results.

\begin{proposition}\label{orderppflaarith}
Let $R$ be an arithmetical ring and for $1\leq i\leq n$, $A_i,B_i,C,D\lhd R$ be finitely generated ideals.  The following are equivalent:
\begin{enumerate}
\item $ C|x\wedge xD=0\leq \sum_{i=1}^n A_i|x\wedge xB_i=0$
\item $\bigwedge_{i=1}xA_i=0+B_i|x\leq xC=0+D|x$
\item $C\subseteq \sum_{i=1}^n(CD:B_i)\cap A_i +CD$
\end{enumerate}
\end{proposition}

\begin{proof}
Prest's duality implies that $(1)$ and $(2)$ are equivalent. We now show that $(1)$ and $(3)$ are equivalent. Let $\mfrak{p}\lhd R$ be a prime ideal and $N$ an $R_{\mfrak{p}}$-module. For any finitely generated ideal $K$, $xK=0$ is equivalent to $xKR_{\mfrak{p}}=0$ and $K|x$ is equivalent to $KR_{\mfrak{p}}|x$. So, since $\phi\leq \psi$ if and only if $\phi(N)\subseteq  \psi(N)$ for all indecomposable pure-injective modules and all indecomposable pure-injective $R$-modules are $R_{\mfrak{p}}$-modules for some prime ideal $\mfrak{p}$, (1) holds if and only if $ CR_{\mfrak{p}}|x\wedge xDR_{\mfrak{p}}=0\leq \sum_{i=1}^n A_iR_{\mfrak{p}}|x\wedge xB_iR_{\mfrak{p}}=0$ in $\pp_{R_{\mfrak{p}}}^1$ for all prime ideals $\mfrak{p}\lhd R$.

By \ref{locidealeq}, $(3)$ holds if and only if $CR_{\mfrak{p}}\subseteq \sum_{i=1}^n(CR_{\mfrak{p}}DR_{\mfrak{p}}:B_iR_{\mfrak{p}})\cap A_iR_{\mfrak{p}} +CR_{\mfrak{p}}DR_{\mfrak{p}}$ for a prime ideals $\mfrak{p}$. Therefore, it is enough to check that $(1)\Leftrightarrow(3)$ locally, that is for valuation rings. Thus we may assume that $D=dR, C=cR$. Then $c+CD\in R/CD$ is a free realisation of $xD=0\wedge C|x$. So $(1)$ holds if and only if $c+CD$ satisfies $\sum_{i=1}^nA_i|x\wedge xB_i=0$. For $A,B\lhd R$ finitely generated, $xB=0$ defines $(CD:B)+CD$ in $R/CD$ and $A|x$ defines $A+CD$ in $R/CD$. Thus $c\in R/CD$ satisfies $\sum_{i=1}^n A_i|x\wedge xB_i=0$ if and only if (3) holds.
\end{proof}

The following remark is easy to show for valuation rings and hence follows for arithmetical rings using \ref{locidealeq}.
\begin{remark}\label{idealquotientarith}
Let $R$ be an arithmetical ring and $A,B\lhd R$ finitely generated ideals. Then $(AB:B)=A+(0:B)$.
\end{remark}


\begin{lemma}\label{Zgbasicmanip}
Let $R$ be an arithmetical ring and let $A,B,C,D\lhd R$ be finitely generated ideals. Then
\[C|x\wedge xD=0\ \leq\ xA=0+B|x\]
if and only if
$AC\subseteq AB+CD$. This is further equivalent to $(B:C)+(D:A)=R$ when $R$ is a Pr\"ufer domain.
\end{lemma}

\begin{proof}
Let $B=\sum_{i=1}^nb_iR$. Then $xA=0+B|x$ is equivalent to \[\sum_{i=1}^nb_iR|x\wedge x0=0+R|x\wedge xA=0.\] So, by \ref{orderppflaarith}, \[C|x\wedge xD=0\ \leq\ xA=0+B|x\] is equivalent to
\[C\subseteq \sum_{i=1}^n(CD:0)\cap b_iR+(CD:A)+CD=B+(CD:A).\] For any ring $R$, it is easy to see that $C\subseteq B+(CD:A)$ implies $AC\subseteq AB+CD$. Suppose $AC\subseteq AB+CD$. Then, since $C\subseteq (AC:A)$, \[C\subseteq (AB+CD:A)=(AB:A)+(CD:A)=(B:A)+(0:A)+(CD:A)=(B:A)+(CD:A).\] The first equality follows from \ref{arithdef} and the second two follow from \ref{idealquotientarith}.

Now suppose that $R$ is a Pr\"ufer domain and $AC\subseteq AB+CD$. Then, by \ref{arithdef}, $R=(AC:AC)\subseteq (AB:AC)+(CD:AC)$. If $A=0$ or $C=0$ then $(D:A)+(B:C)=R$. So suppose that $A,C\neq 0$. By \ref{idealquotientarith}, $(AB:AC)=((AB:A):C)=(B:C)$ and $(CD:AC)=(D:A)$. Therefore $(B:C)+(D:A)=R$.

Now suppose that $(B:C)+(D:A)=R$. Take $u\in (B:C)$ such that $1-u\in (D:A)$. Let $c\in C$ and $a\in A$. Then $ac=uac+(1-u)ac\in AB+CD$. So, when $R$ is a Pr\"ufer domain,  $AC\subseteq AB+CD$ is equivalent to $(B:C)+(D:A)=R$ as required.
\end{proof}

We are now able to give a nice basis of open sets for the Ziegler spectra of arithmetical rings similar to the bases available for B\'ezout domains.

\begin{lemma}
Let $R$ be an arithmetical ring. The sets
\[\left(\nf{C|x\wedge xD=0}{xA=0+B|x}\right)\] where $A,B,C,D\lhd R$ are finitely generated ideals are a basis of open sets for $\Zg(R)$. 
%
\end{lemma}
\begin{proof}
Let $S$ be a ring and $\phi_i,\psi_j$ pp-$1$-formulae for $1\leq i\leq n$ and $1\leq j\leq m$. Then
\[\left(\nf{\sum_{i=1}^n\phi_i}{\bigwedge_{j=1}^m\psi_j}\right)=\bigcup_{\substack{1\leq i\leq n \\ 1\leq j\leq m}} \left(\nf{\phi_i}{\psi_j}\right).\] So the result follows from \ref{Specialform}. 
\end{proof}

We now switch to working with domains. Various parts of the theory below will go through with modifications of the statements and proofs for non-domains if one replaces $\Gamma(R)_\infty^+$ with the monoid of finitely generated ideals of $R$. However, our final goal is to use the Jaffard-Kaplansky-Ohm Theorem  to transfer results from B\'ezout domains to Pr\"ufer domains and this theorem is only available in the domain case.

%

%

It follows from \ref{Specialform}, that a pp-$1$-type over a Pr\"ufer domain is determined by the formulae of the form
\[xA=0+B|x,\]
where $A,B\lhd R$ are finitely generated ideals, that it contains.

Since $\pp_R^1$ is distributive when $R$ is a Pr\"ufer domain, a pp-$1$-type $p$ is irreducible if and only if for all $\phi_1,\phi_2\notin p$, $\phi_1+\phi_2\notin p$. Therefore, an irreducible pp-$1$-type is determined by the formulae of the form $xA=0$ and $B|x$, where $A,B\lhd R$ are finitely generated, that it contains.

Let $p$ be an irreducible pp-$1$-type, define \[I(p):=\{a\in R \st xa=0\in p\}\ \ \ \text{ and } \ \ \ J(p):=\{b\in R \st b|x\notin p\}.\]

\begin{definition}
\noindent
\begin{enumerate}[(i)]
\item We say a pair of weakly prime ideals $(I,J)$ is \textbf{admissible} if either $I^\#\subseteq J^\#$ or $J^\#\subseteq I^\#$.
\item For $(I,J)$ an admissible pair of weakly prime ideals, let $p(I,J)$ be the set of pp-$1$-formulae
\[\bigwedge_{i=1}^n(xA_i=0+B_i|x)\] where $A_i,B_i\lhd R$ are finitely generated ideals such that for each $1\leq i\leq n$ either $A_i\subseteq I$ or $B_i\nsubseteq J$.
\end{enumerate}
\end{definition}

When $R$ is a valuation domain, all proper ideals are weakly prime and all pairs of proper ideals are admissible.
We derive the next lemma from the valuation domain case.

\begin{lemma}\label{addpairlemma}
Let $R$ be a Pr\"ufer domain.
\begin{enumerate}
\item Let $p$ be an irreducible pp-$1$-type. Then $(I(p),J(p))$ is an admissible pair of weakly prime ideals. Moreover, for $A,B\lhd R$ finitely generated ideals, $xA=0\in p$ if only if $A\subseteq I(p)$ and $B|x\notin p$ if and if $B\subseteq J(p)$. Hence $p(I(p),J(p))=p$.

\item Let $(I,J)$ be an admissible pair of weakly prime ideals. Then $p(I,J)$ is an irreducible pp-$1$-type. Moreover, $I(p(I,J))=I$ and $J(p(I,J))=J$.
\end{enumerate}
\end{lemma}
\begin{proof}
\noindent
(1) There exists an indecomposable pure-injective $R$-module $N$ and $m\in N\backslash\{0\}$ such that $m$ has pp-type $p$ in $N$. By \ref{Zglocglob}, $N$ is the restriction of an $R_{\mfrak{p}}$-module for some prime ideal $\mfrak{p}\lhd R$. Let $p'$ be the pp-type of $m$ in $N$ as an $R_{\mfrak{p}}$-module. Then $I(p)=I(p')\cap R$ and $J(p)=J(p')\cap R$. Hence, by \ref{formwkprime}, $I(p)$ and $J(p)$ are weakly prime.  By \ref{hashandloc}(2), $I^\#,J^\#\subseteq \mfrak{p}$. Therefore, since $R$ is a Pr\"ufer domain, $I^\#$ and $J^\#$ are comparable by inclusion. So $(I(p),J(p))$ is an admissible pair of weakly prime ideals.

Let $A=\sum_{i=1}^na_iR$. Then $xA=0\in p$ if and only if $\bigwedge_{i=1}^nxa_i=0\in p$ if and only if $xa_i=0\in p$ for all $1\leq i\leq n$. Hence $xA=0\in p$ if and only if $a_i\in I(p)$ for all $1\leq i\leq n$.

Let $B=\sum_{i=1}^nb_iR$. Then $B|x\notin p$ if and only if $\sum_{i=1}^n b_i|x\notin p$. Since $p$ is irreducible, this is true exactly when $b_i|x\notin p$ for all $1\leq i\leq n$. Therefore $B|x\notin p$ if and only if $B\subseteq J(p)$.

\noindent
(2) Let $\mfrak{p}\lhd R$ be a prime ideal such that $\mfrak{p}\supseteq I^\#, J^\#$. By \ref{formwkprime}, $I=IR_{\mfrak{p}}\cap R$ and $J=JR_{\mfrak{p}}\cap R$. By \cite[3.4]{EH}, noting that the consistency conditions in \cite[3.2]{EH} are empty for valuation domains, there exists an indecomposable pure-injective $R_\mfrak{p}$-module $N$ and $m\in N\backslash\{0\}$ such that $ma=0$ if and only if $a\in IR_{\mfrak{p}}$, and, $a|m$ if and only if $a\notin JR_{\mfrak{p}}$. So $I(\pp_N(m))=IR_{\mfrak{p}}\cap R=I$ and $J(\pp_N(m))=JR_{\mfrak{p}}\cap R = J$. Since $\pp_N(m)$ is irreducible, it now follows from \ref{Specialform} and the definition of $p(I,J)$ that $\pp_N(m)=p(I,J)$.
\end{proof}

\begin{definition}
Let $(I,J)$ be an admissible pair of weakly prime ideals of $R$. Let $N(I,J)$ denote the pure-injective hull of the irreducible pp-$1$-type $p(I,J)$.
\end{definition}

\begin{cor}\label{ALLNIJ}
Let $R$ be a Pr\"ufer domain. All indecomposable pure-injective $R$-modules are of the form $N(I,J)$ for some admissible pair of weakly prime ideals $(I,J)$.
\end{cor}

We get the following as a corollary to the proof of \ref{addpairlemma}.

\begin{cor} \label{AttIJ}
Let $R$ be a Pr\"ufer domain, $\mfrak{p}\lhd R$ a prime ideal and $(I,J)$ an admissible pair of weakly prime ideals.
\begin{enumerate}
\item Then $N(I,J)\in\Zg(R_\mfrak{p})$ if and only if $\mfrak{p}\supseteq I^\#,J^\#$. In particular, $\Att N=I^\#\cup J^\#$.
\item For all prime ideals $\mfrak{p}\supseteq I^\#\cup J^\#$, $N(I,J)$ is the restriction to $R$ of the $R_\mfrak{p}$-module $N(IR_\mfrak{p},JR_\mfrak{p})$.
\end{enumerate}
\end{cor}

Next we consider when two admissible pairs of weakly prime ideals are realised in the same indecomposable pure-injective module. For $(I_1,J_1)$, $(I_2,J_2)$ pairs of weakly prime ideals, we say $(I_1,J_1)\sim (I_2,J_2)$ if there exists a finitely generated ideal $K\lhd R$ such that one of the following holds:
\begin{enumerate}[(i)]
\item $(I_1:K)=I_2$ and $J_1=(J_2:K)$; or
\item $I_1=(I_2:K)$ and $J_2=(J_1:K)$.
\end{enumerate}

\begin{lemma}\label{equpairs}
Let $R$ be a Pr\"ufer domain and let $(I_1,J_1), (I_2,J_2)$ be admissible pairs of weakly prime ideals. Then $(I_1,J_1)\sim (I_2,J_2)$ if and only if $N(I_1,J_1)\cong N(I_2,J_2)$.
\end{lemma}
\begin{proof}
Suppose $R$ is a valuation domain. Note that for proper ideals $I_1,I_2\lhd R$ and $K\lhd R$ a finitely generated ideal, $(I_1:K)=I_2$ if and only if $I_1=I_2K$ and $K\neq 0$. So, since all finitely generated ideals of a valuation domain are principal, the lemma follows for valuation domains from \cite[11.11]{Serialrings}.

Now let $R$ be a Pr\"ufer domain. Suppose  $N(I_1,J_1)\cong N(I_2,J_2)$. Then, by \ref{AttIJ},
\[I_1^\#\cup J_1^\#=\Att N(I_1,J_1)=\Att N(I_2,J_2)=I_2^\#\cup J_2^\#.\]
It is easy to see that $(I_1,J_1)\sim (I_2,J_2)$ implies $I_1^\#=I_2^\#$ and $J_1^\#=J_2^\#$. Hence if either side of the equivalence holds then $I_1^\#\cup J_1^\#=I_2^\#\cup J_2^\#$. Let $\mfrak{p}:=I_1^\#\cup J_1^\#=I_2^\#\cup J_2^\#$. For $i=1,2$, $N(I_i,J_i)$ is the restriction of the $R_\mfrak{p}$-module $N(I_iR_\mfrak{p},J_iR_\mfrak{p})$ to $R$. Hence $N(I_1,J_1)\cong N(I_2,J_2)$ if and only if $N(I_1R_\mfrak{p},J_1R_\mfrak{p})\cong N(I_2R_\mfrak{p},J_2R_\mfrak{p})$. By \ref{formwkprime}, for $i=1,2$, $I_i=I_iR_\mfrak{p}\cap R$ and $J_i=J_iR_\mfrak{p}\cap R$. Thus, by \ref{idealquotientarith}, $(I_1,J_1)\sim (I_2,J_2)$ if and only if $(I_1R_\mfrak{p},J_1R_\mfrak{p})\sim (I_2R_\mfrak{p},J_2R_\mfrak{p})$. Hence, the claimed equivalence follows from the valuation domain case.
%
%
%
%
%
%
\end{proof}

In \cite{bezwidth}, in the B\'ezout domain case, the equivalence relation $\sim$ is defined in terms of $J_a:=J_{aR}$ where $J\lhd R$ is a weakly prime ideal and $a\in R$. Our definition of $\sim$ side steps the need for this construction. 

Given an admissible pair of weakly prime ideals,
the next lemma shows how to manufacture all the other pairs which are $\sim$-equivalent to it. Moreover it shows how to define $\sim$ in the style of \cite{bezwidth} for Pr\"ufer domains.



\begin{lemma}\label{admissiblepairshift}
If $(I,J)$ is an admissible pair of weakly prime ideals and $K\nsubseteq I$ (respectively $K\nsubseteq J$) then so is $((I:K),J_K)$ (respectively $(I_K,(J:K))$). Moreover, if  $K\nsubseteq I$ (respectively $K\nsubseteq J$) then $(I,J)\sim ((I:K),J_K)$ (respectively $(I,J)\sim (I_K:(J:K))$).
\end{lemma}
\begin{proof}
By the comment just below \ref{eqwkprime}, $(I:K)$ is weakly prime and  $(I:K)^\#=I^\#$. By \ref{InverseIdealQuotient}, $J_K$ is weakly prime and $(J_K)^\#=J^\#$. Thus if $(I,J)$ is admissible then so is $((I:K),J_K)$. By \ref{InverseIdealQuotient}, $(J_K:K)=J$, so $(I,J)\sim ((I:K), J_K)$.
\end{proof}

\begin{lemma}\label{toponpoint}
Let $R$ be a Pr\"ufer domains and let $(I,J)$ be an admissible pair of weakly prime ideals. Then \[N(I,J)\in \left(\nf{C|x\wedge xD=0}{xA=0+B|x}\right)\] if and only if there exists an admissible pair $(I',J')$ with $(I,J)\sim (I',J')$ such that $C\nsubseteq J'$, $B\subseteq J'$, $D\subseteq I'$ and $A\nsubseteq I'$.
\end{lemma}
\begin{proof}
This follows from \ref{addpairlemma} and \ref{equpairs}.
\end{proof}

We are ready to prove our transfer theorem.
The theorem has essentially already been established. We just need to translate what we have shown into the value group.

\begin{theorem}\label{BeztoPruf}
Let $R,S$ be Pr\"ufer domains. If $t:\Gamma(S)\rightarrow \Gamma(R)$ is an isomorphism of $\ell$-groups then there exists a lattice isomorphism $\lambda:\pp_S^1\rightarrow \pp_R^1$, such that $\lambda(xA=0)$ is $xt(A)=0$ and $\lambda(A|x)$ is $t(A)|x$, and a homeomorphism of topological spaces $\rho:\Zg(R)\rightarrow \Zg(S)$ such that for all $\phi,\psi\in\pp_S^1$, $(\lambda\phi/\lambda\psi)=\rho^{-1}(\phi/\psi)$.
\end{theorem}
\begin{proof}
Let $t:\Gamma(S)\rightarrow \Gamma(R)$ be an isomorphism of $\ell$-groups. We extend $t$ by mapping the zero ideal to the zero ideal.
By \ref{Specialform}, every $\phi\in \pp_S^1$ is equivalent to one of the form $\sum_{i=1}^nxA_i=0\wedge B_i|x$ where $A_i,B_i$ are finitely generated ideals. Define $\lambda(\sum_{i=1}^nxA_i=0\wedge B_i|x)$ to be $\sum_{i=1}^nxt(A_i)=0\wedge t(B_i)|x$. To see that this map is well-defined and order preserving, it is enough to note that the operations on finitely generated ideals involved in \ref{orderppflaarith} (3) are preserved by $t$.

Let
\[\widehat{t}:\Filters(\Gamma(R)_\infty^+)\rightarrow \Filters(\Gamma(S)_\infty^+), \ \ \ \mcal{F}\mapsto \{A\in \Gamma(S)^+\st t(A)\in \mcal{F}\}\cup\{\infty\}\] where for $\Gamma$ an $\ell$-group, $\Filters(\Gamma_\infty^+)$ denotes the set of filters of $\Gamma_\infty^+$. The bijection $\widehat{t}$ restricts to a bijection between prime filters. Clearly, for $\mcal{F}\in \Filters(\Gamma(R))$ and $K\in \Gamma(R)^+$, $\widehat{t}(\mcal{F}:K)=(\widehat{t}(\mcal{F}):t^{-1}K)$ and if $\mcal{F}$ is a prime filter then $\widehat{t}(\mcal{F}^\#)=\widehat{t}(\mcal{F})^\#$.

As remarked in section \ref{SPreI}, ideals of $R$ correspond to filters of $\Gamma(R)_\infty^+$ and weakly prime ideals of $R$ correspond to prime filters of $\Gamma(R)_\infty^+$ via the map $I\mapsto v(I)$. Hence $v^{-1}\widehat{t}v:\Ideals(R)\rightarrow \Ideals(S)$ is bijective and restricts to a bijective map between weakly prime ideals. For $I\lhd R$, a weakly prime ideal $v(I^\#)=v(I)^\#$. Therefore $v^{-1}\widehat{t}v(I^\#)=v^{-1}\widehat{t}v(I)^\#$. Therefore $v^{-1}\widehat{t}v$ induces a bijective map between the admissible pairs of weakly prime ideals of $R$ and the admissible pairs of weakly prime ideals of $S$. For $I\lhd R$, $v^{-1}\widehat{t}v(I:K)=v^{-1}\widehat{t}(vI:K)=v^{-1}(\widehat{t}vI:t^{-1}K)=(v^{-1}\widehat{t}v I:t^{-1}K)$. Hence $v^{-1}\widehat{t}v$ respects and reflects the $\sim$-equivalence classes of admissible pairs of weakly prime ideals. For $N\in\Zg(R)$, define $\rho(N)=N(v^{-1}\widehat{t}vI,v^{-1}\widehat{t}vJ)$ where $N=N(I,J)$; recall, \ref{ALLNIJ}, such an $(I,J)$ always exists. By \ref{equpairs}, this map is well-defined and applying the same arguments using $t^{-1}$ in place of $t$ we see that $\rho$ is bijective.

We now just need to show that $\rho$ is compatible with $\lambda$. By symmetry, it is enough to show that for all $\phi,\psi\in \pp_S^1$, $N\in \left(\nf{\lambda\phi}{\lambda\psi}\right)$ implies $\rho(N)\in\left(\nf{\phi}{\psi}\right)$. Let $A,B,C,D\lhd S$ be finitely generated ideals and $N\in\left(\nf{\lambda(C|x\wedge xD=0)}{\lambda(xA=0+B|x)}\right)$.  By definition $\left(\nf{\lambda(C|x\wedge xD=0)}{\lambda(xA=0+B|x)}\right)=\left(\nf{t(C)|x\wedge xt(D)=0}{xt(A)=0+t(B)|x}\right)$. So, by \ref{ALLNIJ} and \ref{toponpoint}, there exists $(I,J)$ an admissible pair of weakly prime ideals of $R$ with $N=N(I,J)$ such that $t(A)\nsubseteq I$, $t(D)\subseteq I$, $t(B)\subseteq J$ and $t(C)\nsubseteq J$. Then $A\nsubseteq v^{-1}\widehat{t}vI$, $D\subseteq v^{-1}\widehat{t}vI$, $B\subseteq v^{-1}\widehat{t}vJ$ and $C\nsubseteq v^{-1}\widehat{t}vJ$. Hence $\rho(N)=N(v^{-1}\widehat{t}vI,v^{-1}\widehat{t}vJ)\in\left(\nf{\lambda(C|x\wedge xD=0)}{\lambda(xA=0+B|x)}\right)$.
Now let $\phi,\psi\in \pp_S^1$. By \ref{Specialform}, we may assume $\phi=\sum_{i=1}^n\phi_i$ where each $\phi_i$ is of the form $C|x\wedge xD=0$ and $\psi=\bigwedge_{j=1}^m\psi_j$ where each $\psi_j$ is of the form $xA=0+B|x$.  Since $\lambda$ is a lattice homomorphism,
\[\left(\nf{\lambda(\sum_i\phi_i)}{\lambda(\bigwedge_j\psi_j)}\right)=\left(\nf{\sum_i\lambda(\phi_i)}{\bigwedge_j\lambda(\psi_j)}\right)=\bigcup_{i,j}\left(\nf{\lambda(\phi_i)}{\lambda(\psi_j)}\right).\]
Thus, it follows from what we have already proved that
\[(\lambda\phi/\lambda\psi)=\bigcup_{i,j}\left(\nf{\lambda(\phi_i)}{\lambda(\psi_j)}\right)=\bigcup_{i,j}\rho^{-1}\left(\nf{\phi_i}{\psi_j}\right)=\rho^{-1}\left(\nf{\phi}{\psi}\right).\qedhere\]
\end{proof}

\begin{remark}
The results in section $5$ and $7$ of \cite{bezwidth} about pp-types over B\'ezout domains, appropriately translated, also hold for Pr\"ufer domains.
\end{remark}

We record the translation of \cite[5.2]{bezwidth} and \cite[5.4]{bezwidth} as \ref{typesintofun} below.

For $\Gamma$ an $\ell$-group, write $\widehat{\Gamma^+_\infty}$ for the set of lattice ideals\footnote{In \cite{bezwidth}, lattice ideals were referred to as cofilters to avoid confusion with ideals of a ring.} of $\Gamma^+_\infty$. Recall that a \textbf{lattice ideal} is a downwards closed non-empty subset of $\Gamma_\infty^+$ closed under taking finite joins.

Let $p\subseteq \pp_R^1$ be a pp-type i.e. a filter in $\pp_R^1$. For each $B\in \Gamma(R)^+_\infty$ define
\[F_p(B):=\{A\in \Gamma(R)_\infty^+ \st A|x+xB=0\in p\}\] where $\infty|x$ is identified with the formula $x=0$ and $x\infty=0$ is identified with the formula $x=x$.

We use small letters $a,b,b'$ in this theorem to emphasise that we are working in $\Gamma(R)$ as an $\ell$-group i.e. when we write $b+b'$ we do not mean the sum of two ideals of the ring $R$ but the addition in $\Gamma(R)$.
\begin{theorem}\label{typesintofun}
Let $R$ be a Pr\"ufer domain. For all pp-types $p\subseteq \pp_R^1$ and $b\in\Gamma(R)^+_\infty$, $F_p(b)$ is a lattice ideal of $\Gamma(R)_\infty^+$. The function
$F_p:\Gamma(R)_\infty^+\rightarrow \widehat{\Gamma(R)^+_\infty}$ has the following properties.
\begin{enumerate}
\item $F(0)=\Gamma(R)_\infty^+$ if and only if $x=0\in p$.
\item $F(\infty)=\Gamma(R)_\infty^+$.
\item For all $a,b,b'\in \Gamma(R)_\infty^+$, if $a\in F(b)$ then $a+b'\in F(b+b')$.
\item For all $a,b,b'\in \Gamma(R)_\infty^+$, if $a\in F(b+b')$ and $a\wedge b'=R$ then $a\in F(b)$.
\item For all $a,b\in \Gamma(R)_\infty^+$, $F(a\wedge b)=F(a)\cap F(b)$.
\end{enumerate}
Conversely, if $F:\Gamma(R)_\infty^+\rightarrow \widehat{\Gamma(R)^+_\infty}$ satisfies (2)-(5) then there is a pp-type $p$ such that $F_p=F$.
\end{theorem}

\section{Preliminaries II}\label{PreII}

\noindent
Throughout, let $R$ be an arbitrary ring. Various ordinal valued dimensions on lattices of pp formulae are considered in model theory of modules and in representation theory. A uniform way of defining many such dimensions on bounded modular lattices is given in \cite{PrestBluebook}. Let $\mcal{L}$ be a non-empty\footnote{That the class is non-empty ensures that it contains the lattice of size $1$.} class of lattices closed under sublattices and quotient lattices. For $L$ a bounded modular lattice, let $\sim_\mcal{L}$ be the smallest congruence relation on $L$ such that for all $a\leq b\in L$, if $[a,b]$ is in $\mcal{L}$ then $a\sim_\mcal{L} b$. For $a,b\in L$, $a\sim_\mcal{L} b$ if and only if there exists $a\wedge b =:c_0\leq c_1\leq \cdots\leq c_n:=a\vee b$ such that each interval $[c_i,c_{i+1}]$ belongs to $\mcal{L}$ \cite[10.3]{PrestBluebook}.

Following \cite{PrestBluebook}, we define, by ordinal induction, a sequence of lattices $L_{\mcal{L},\alpha}$, congruence relations $\sim_{\mcal{L},\alpha}$ and quotient maps $\pi_{\mcal{L},\alpha}:L\rightarrow L_{\mcal{L},\alpha}$ for $\alpha$ an ordinal. Let $L_{\mcal{L},0}:=L$ and let $\sim_0$ be the trivial congruence relation. Define $L_{\mcal{L},\alpha+1}:=L_{\mcal{L},\alpha}/\sim_{\mcal{L}}$ and $\pi_{\mcal{L},\alpha+1}$ to be the composition of $\pi_\alpha$ with the quotient map $L_{\mcal{L},\alpha}\rightarrow L_{\mcal{L},\alpha+1}$. For $\alpha$ a successor ordinal, define $\sim_{\mcal{L},\alpha}$ to be the congruence kernel of $\pi_\alpha$, i.e. $a\sim_{\mcal{L},\alpha}b$ if and only if $\pi_{\mcal{L},\alpha}(a)=\pi_{\mcal{L},\alpha}(b)$. For $\lambda$ a limit ordinal, define $\sim_{\mcal{L},\lambda}:=\bigcup_{\alpha<\lambda}\sim_{\mcal{L},\alpha}$, $L_{\mcal{L},\lambda}:=L/\sim_{\mcal{L},\lambda}$ and $\pi_{\mcal{L},\lambda}:=L\rightarrow L_{\mcal{L},\lambda}$ the quotient map. Let $\alpha$ be the least ordinal such that $L_{\mcal{L},\alpha}$ is the one-point lattice. It follows from the definition that $\alpha$ is always a successor ordinal (unless $L$ is the one-point lattice). Define the $\mcal{L}$-dimension of $L$ to be $\mcal{L}\text{-}\dim(L):=\alpha-1$.

\begin{remark}
Let $\mcal{L}$ be a non-empty class of lattices closed under sublattices and quotients and let $L$ be a modular lattice. An interval $[a,b]$ in $L$ has $\mcal{L}\text{-}\dim\leq \alpha$ if and only if $a\sim_{\mcal{L},\alpha+1} b$.
\end{remark}

Every closed subset $\mcal{C}$ of $\Zg(R)$ gives rise to a congruence relation $\approx_\mcal{C}$ on $\pp_R^n$ for each $n\in\N$ by setting $\phi\approx_\mcal{C}\psi$ whenever $\phi(N)=\psi(N)$ for all $N\in \mcal{C}$. In general, there are many  congruence relations on $\pp_R^n$ not induced by closed subsets of $\Zg(R)$.
The next proposition follows easily from the results in \cite[\S13]{PSL}. However, we weren't able to find it explicitly stated anywhere so we give some indication of its proof.

\begin{proposition}\label{closedsetforLdim}
Let $\mcal{L}$ be a non-empty class of lattices closed under sublattices and quotient lattices. The closed subset \[\mcal{C}_\mcal{L}:=\bigcap_{\substack{\phi,\psi\in \pp_R^1 \\ \phi\sim_\mcal{L}\psi}} \Zg(R)\backslash\left(\nf{\phi}{\psi}\right)\] is such that $\sim_\mcal{L}$ is equal to $\approx_{\mcal{C}_\mcal{L}}$.
\end{proposition}
\begin{proof}
In this proof we take our notation concerning functor categories from \cite{PSL}. Following \cite[\S13]{PSL}, let $\mcal{S}_\mcal{L}$ be the Serre subcategory of $(\mod\text{-}R, \Ab)^{fp}$ generated by those $F\in (\mod\text{-}R, \Ab)^{fp}$ whose lattice of finitely presented subobjects are in $\mcal{L}$. By \cite[13.1.2]{PSL}, $S_\mcal{L}$ consists exactly of those $F\in (\mod\text{-}R, \Ab)^{fp}$ whose lattice of finitely presented subobjects have $\mcal{L}$-dimension $0$. Let \[\mcal{C}_{\mcal{L}}:=\{N\in \Zg(R) \st \overrightarrow{F}N=0 \text{ for all }F\in\mcal{S}_\mcal{L}\}. \] By \cite[12.4.1]{PSL}, $F\in \mcal{S}_\mcal{L}$ if and only if $\overrightarrow{F}N=0$ for all $N\in \mcal{C}_{\mcal{L}}$.

For $\phi\geq \psi \in \pp_R^1$, the lattice of finitely presented subobjects of $F_{\nf{\phi}{\psi}}$, the functor which sends $M\in\mod\text{-}R$ to $\phi(M)/\psi(M)$, is isomorphic to the interval $[\psi,\phi]$. Moreover, for $N\in\Mod\text{-}R$, $\phi(N)=\psi(N)$ if and only if $\overrightarrow{F_{\nf{\phi}{\psi}}}N=0$. Thus $\mcal{C}_\mcal{L}$ is a Ziegler closed subset with the required properties. Standard arguments for reducing to pp formulae in one variable now show that $N\in \mcal{C}_\mcal{L}$ if and only if $N\notin \left(\nf{\phi}{\psi}\right)$ for all $\phi,\psi\in \pp_R^1$ with $\phi\sim_\mcal{L}\psi$.
%
%
%
%
%
%
%
%
%
\end{proof}

We are principally interested in the two most prominent of these dimensions: m-dimension and breadth. We write $\two$ for the class of lattices of size $\leq 2$ and $\Ch$ for the class of total orders. The dimension associated to $\two$, i.e. $\two$-dimension\footnote{Note that this is not equal to ``$2$-dimension'' in the sense of \cite{PunSDstring}.}, is called \textbf{m-dimension}. We will write $\mdim L$ for this dimension rather than $\two\text{-}\dim L$. The dimension associated to $\Ch$ is called \textbf{breadth}. We will write $\br L$ for this dimension rather than $\Ch\text{-}\dim L$.

We will not use the definition of the width of a modular lattice in this paper, see \cite[\S 7]{Zieglermodules} for its definition. Since many of the results in the literature are given in terms of existence of width, we record here that width and breadth coexist.

\begin{theorem}\cite[10.7]{PrestBluebook}
The width of a modular lattice $L$ is defined if and only if the breadth of $L$ is defined.
\end{theorem}

Unlike breadth, the actual value of width is not particularly meaningful.
The purpose of width (and hence breadth) is to detect the existence of superdecomposable pure-injective modules. Recall that a module is \textbf{superdecomposable} if it has no non-zero indecomposable direct summands. There are currently no known examples of rings without superdecomposable pure-injective modules such that $\pp_R^1$ does not have width (and hence breadth).
From here on we will write all results in terms of breadth rather than width, independently
of how they were originally stated.

\begin{theorem}\cite[7.1]{Zieglermodules}
If there exists a superdecomposable pure-injective $R$-module then $\pp_R^1$ does not have breadth. Moreover, if $R$ is countable, then the converse holds.
\end{theorem}

We say $\psi\notin p$ is \textbf{large} in $p$ if, for all $\theta_1,\theta_2\geq \psi$, if $\theta_1,\theta_2\notin p$, then there is $\sigma\in p$ such that $\sigma\wedge \theta_1+\sigma\wedge\theta_2\notin p$. A pp-type $p$ is \textbf{superdecomposable} if and only if no $\psi\notin p$ is large in $p$.

\begin{theorem}\cite[7.6]{Zieglermodules}\label{sdtypetopi}
If $p$ is a superdecomposable pp-type then $N(p)$ is superdecomposable. Moreover, if $M$ is superdecomposable and $m\in M$ is non-zero then $p=\pp_M(m)$ is superdecomposable. In particular, there exists a superdecomposable pure-injective $R$-module if and only if there exists a superdecomposable pure-injective pp-$1$-type.
\end{theorem}

If $\pp_R^1$ is distributive, so in particular, if $R$ is an arithmetical ring, then a pp-$1$-type $p$ is superdecomposable if and only if for all $\psi\notin p$, there exists $\psi\leq \theta_1,\theta_2\notin p$ such that $\theta_1+\theta_2\notin p$.

The \textbf{Cantor-Bendixson derivative} of a topological space $X$, denoted $X'$, is the set of non-isolated points of $X$. For $\alpha$ an ordinal, we define the $\alpha$th Cantor-Bendixson derivative $X^{(\alpha)}$ of $X$ by ordinal induction as follows:
\begin{itemize}
\item $X^{(0)}:=X$;
\item $X^{(\alpha+1)}:=(X^{(\alpha)})'$; and
\item if $\lambda$ is a limit ordinal then $X^{(\lambda)}:=\bigcap_{\alpha<\lambda}X^{(\alpha)}$.
\end{itemize}
The \textbf{Cantor-Bendixson rank} of a point $x\in X$ is the supremum of the ordinals $\alpha$ such that $x\in X^{(\alpha)}$ if such an ordinal exists and $\infty$ otherwise. The \textbf{Cantor-Bendixson rank} (\textbf{CB rank}) of a non-empty space $X$, denoted $\CBrank(X)$, is the supremum of all $\alpha$ such that $X^{(\alpha)}$ is non-empty if it exists and $\infty$ otherwise. If $X$ is compact with CB rank $\alpha$ then $X^{(\alpha)}$ is non-empty and finite.

A ring $R$ satisfies the \textbf{isolation condition} if for every closed subset $X\subseteq \Zg(R)$ and every $N\in X$ which is isolated in $X$ there exists $\phi\geq\psi\in \pp_R^1$ such that $X\cap\left(\nf{\phi}{\psi}\right)=\{N\}$ and $\phi/\psi$ is $X$-minimal i.e. for all $\phi\geq \sigma\geq \psi$ either $\phi(M)=\sigma(M)$ for all $M\in X$ or $\psi(M)=\sigma(M)$ for all $M\in X$. Countable rings satisfy the isolation condition, \cite[10.15]{PrestBluebook}.

\begin{theorem}
If $\pp_R^1$ has m-dimension $\alpha$ then $\Zg(R)$ has CB rank $\alpha$. The converse holds when $R$ satisfies the isolation condition.
\end{theorem}

In this paper, we are in the fortunate position of working with rings whose lattices of pp-$1$-formulae are distributive. Hence we may use the following useful theorem of Puninski.

\begin{theorem}\cite[3.4]{KGdimSer}\label{CBmdimdist}
If $\pp_R^1$ is distributive then $R$ satisfies the isolation condition. Therefore the CB rank of $\Zg(R)$ is equal to the m-dimension of $\pp_R^1$.
\end{theorem}

\section{Dimensions on lattice ordered abelian groups}\label{SDimLgroup}
\noindent
Let $\mcal{L}$ be a non-empty class of lattices closed under sublattices, quotients and \textit{taking opposites}. Let $\Gamma$ be an $\ell$-group. Define $C_\mcal{L}(\Gamma)$ to be the set of $a\in \Gamma$ such that $a\sim_\mcal{L} 0$. Note that $a\sim_\mcal{L} 0$ is equivalent to $[0,a\vee 0]$ and $[0,-(a\wedge 0)]$ having $\mcal{L}\text{-}$dimension $0$ because $[0,-(a\wedge 0)]$ is anti-isomorphic to $[a\wedge 0,0]$ and $\mcal{L}$ is closed under taking opposites. Whenever it won't cause confusion, we will write $C_\mcal{L}$ rather than $C_\mcal{L}(\Gamma)$.


\begin{lemma}\label{Ldimlgroup}
The set $C_\mcal{L}:=C_\mcal{L}(\Gamma)$ is a convex $\ell$-subgroup of $\Gamma$. Moveover, $C_\mcal{L}^+$, the positive cone of $C_\mcal{L}$, is generated as a monoid by the elements $a\in \Gamma^+$ such that $[0,a]\in\mcal{L}$. Hence $C_\mcal{L}$ is generated as a group by these elements.
\end{lemma}
\begin{proof}
Since $(-a)\wedge 0=-(a\vee 0)$ and $(-a)\vee 0=-(a\wedge 0)$, $C_\mcal{L}$ is closed under taking inverses.

Recall that in any modular lattice $L$, if $a\geq a'\geq b'\geq b$ and $[a,b]$ has $\mcal{L}$-dimension $0$ then $[a',b']$ has $\mcal{L}$-dimension $0$. Note that, for all $a,b\in \Gamma$,
\[a\vee 0+b\vee 0\geq (a+b)\vee 0 \ \ \ \text{ and } \ \ \ a\wedge 0+b\wedge 0\leq (a+b)\wedge 0.\]
So, in order to show that $C_{\mcal{L}}$ is a subgroup it is enough to show that if $a,b\in C_{\mcal{L}}$ then the intervals $[0,a\vee 0+b\vee 0]$ and $[0,-(a\wedge 0+b\wedge 0)]$ have $\mcal{L}$-dimension $0$. The interval $[a\vee 0, a\vee 0+b\vee 0]$ has $\mcal{L}$-dimension $0$ because it is isomorphic to $[0,b\vee 0]$. So, since $[0,a\vee 0]$ has $\mcal{L}$-dimension $0$, so does $[0,a\vee 0+b\vee 0]$. Similarly, one can show $[0,-(a\wedge 0+b\wedge 0)]$ has $\mcal{L}$-dimension $0$. So $C_{\mcal{L}}$ is a subgroup.

By definition, for all $a\in C_{\mcal{L}}$, $a\wedge 0\in C_{\mcal{L}}$. So $C_{\mcal{L}}$ is an $\ell$-subgroup by \cite[6.1]{Darnel}. Finally, if $a\leq b\in C_{\mcal{L}}$ and $a\leq c\leq b$ then $0\leq c-a\leq b-a$. Since $b-a\in C_\mcal{L}$ and $(b-a)\vee 0=b-a$, $[0,c-a]=[0,(c-a)\vee 0]$ has $\mcal{L}$-dimension $0$. Thus $c-a\in C_\mcal{L}$. Therefore, since $a\in C_{\mcal{L}}$, $c\in \mcal{C}_\mcal{L}$ as required. So $C_{\mcal{L}}$ is a convex $\ell$-subgroup.

We now prove the final two claims.  Suppose $a\in C_\mcal{L}^+$. There exists $0=a_0<a_1<\ldots <a_n=a$ such that $[a_i,a_{i+1}]\in\mcal{L}$ for $0\leq i\leq n-1$. For each $0\leq i\leq n-1$, $[a_i,a_{i+1}]$ is isomorphic to $[0,a_{i+1}-a_i]$. So $a=\sum_{i=0}^{n-1}(a_{i+1}-a_i)$ and $[0,a_{i+1}-a_i]\in\mcal{L}$ for $0\leq i\leq n-1$ as required. The second claim follows from the first since, $a=a\vee 0+a\wedge 0$ and by definition, $a\vee 0,-a\wedge 0\in C_{\mcal{L}}$. So, since $a=a\vee 0+a\wedge 0$  and $a\vee 0,-a\wedge 0\geq 0$, it is enough to prove the claim for $a>0$.
\end{proof}

\begin{remark}\label{Lpointsdense}
Let $\Gamma$ be a non-trivial $\ell$-group. If $\Gamma_\infty^+$ has $\mcal{L}$-dimension then for all $a\in \Gamma^+$ non-zero there exists $0< c\leq a$ such that $[0,c]\in\mcal{L}$.
\end{remark}
\begin{proof}
If $\Gamma_\infty^+$ has $\mcal{L}$-dimension then for all $0<a$, there exists $0\leq c_1<c_2\leq a$ such that $[c_1,c_2]\in\mcal{L}$. Then $[c_1,c_2]$ is isomorphic to $[0,c_2-c_1]$. Therefore $0<c_2-c_1\leq a$ and $[0,c_2-c_1]\in\mcal{L}$.
\end{proof}

For any $a\in \Gamma^+$, the interval $[a,\infty]$ is order isomorphic to $\Gamma^+_\infty$. Therefore, if $\Gamma_\infty^+$ has $\mcal{L}$-dimension $0$ then $\Gamma_\infty^+\in \mcal{L}$.

\begin{remark}\label{Ldimtrivial}
Let $\Gamma$ be an $\ell$-group. Then $\Gamma^+_\infty$ has m-dimension $0$ (respectively breadth $0$) if and only if $\Gamma$ is the trivial group (respectively totally ordered).
\end{remark}

\begin{lemma}\label{ldimgroupbasecase}
If $\mcal{L}\text{-}\dim\Gamma_\infty^+\geq 1$ then the map from $\Gamma_\infty^+$ to $(\Gamma/C_{\mcal{L}})_\infty^+$ induced by the quotient map from $\Gamma$ to $\Gamma/C_\mcal{L}$ has congruence kernel $\sim_\mcal{L}$. In particular, $(\Gamma/C_\mcal{L})_\infty^+$ is isomorphic to $\Gamma_\infty^+/\sim_{\mcal{L}}$.
\end{lemma}
\begin{proof}
A congruence relation $\sim$ on a lattice $L$ is determined by the pairs $(a,b)\in L^2$ such that $a\leq b$ and $a\sim b$. Thus, it is enough to show that for all $a\leq b\in \Gamma_\infty^+$, $a\sim_{\mcal{L}}b$ if and only if $b-a\in C_{\mcal{L}}$ or $a=b=\infty$. Suppose $a\in \Gamma^+$ and $b=\infty$. Then $[a,\infty]$ is isomorphic to $\Gamma^+_\infty$. So $a\sim_\mcal{L}\infty$ if and only if $\Gamma^+_\infty$ has $\mcal{L}$-dimension $0$. So suppose $a,b\in \Gamma^+$ and $a\leq b$. Now $\mcal{L}\text{-}\dim[a,b]=0$ if and only if $\mcal{L}\text{-}\dim[0,b-a]=0$. By definition, $\mcal{L}\text{-}\dim[0,b-a]=0$ if and only if $b-a\in C_\mcal{L}$.
%
\end{proof}

\begin{definition}
We define $C_{\mcal{L},\alpha}(\Gamma)$ for each $\alpha\in \textnormal{Ord}$. Let $C_{\mcal{L},0}(\Gamma)=\{0\}$. Define $C_{\mcal{L},{\alpha+1}}(\Gamma)$ to be the kernel of the map from $\Gamma\rightarrow \Gamma/C_{\mcal{L},\alpha}(\Gamma)$ composed with the map from $\Gamma/C_{\mcal{L},\alpha}(\Gamma)$ to $(\Gamma/C_{\mcal{L},\alpha}(\Gamma))/C_\mcal{L}(\Gamma/C_{\mcal{L},\alpha}(\Gamma))$. If $\lambda$ is a limit ordinal then define $C_{\mcal{L},\lambda}(\Gamma):=\bigcup_{\alpha<\lambda}C_{\mcal{L},\alpha}(\Gamma)$.
\end{definition}

Whenever $\Gamma$ is clear from the context, we will write $C_{\mcal{L},\alpha}$ rather than $C_{\mcal{L},\alpha}(\Gamma)$.


\begin{proposition}\label{dimforgroup1}
For $\alpha\leq \mcal{L}\text{-}\dim\Gamma_\infty^+$, the map induced from $\Gamma_\infty^+$ to $(\Gamma/C_{\mcal{L},\alpha})_\infty^+$ by the quotient map from $\Gamma$ to $\Gamma/C_{\mcal{L},\alpha}$ has congruence kernel $\sim_{\mcal{L},\alpha}$. In particular, for $\alpha\leq \mcal{L}\text{-}\dim\Gamma_\infty^+$,
\[(\Gamma/C_{\mcal{L},\alpha})_\infty^+\cong \Gamma_\infty^+/\sim_{\mcal{L},\alpha}.\] Moreover,
\begin{enumerate}[(i)]
\item $\mcal{L}\text{-}\dim\Gamma_\infty^+$ is the least ordinal $\alpha$ such that $(\Gamma/C_{\mcal{L},\alpha})_\infty^+$ has $\mcal{L}$-dimension $0$, and,
\item if $\two\subseteq \mcal{L}$ then $\Gamma_\infty^+$ has $\mcal{L}$-dimension if and only if there exists $\alpha\in \Ord$ such that $C_{\mcal{L},\alpha}=\Gamma$.
\end{enumerate}
\end{proposition}
\begin{proof}
We write $\sim_\alpha$ instead of $\sim_{\mcal{L},\alpha}$ and $C_\alpha$ instead of $C_{\mcal{L},\alpha}$. As in \ref{ldimgroupbasecase}, in order to prove that $\sim_\alpha$ is the congruence kernel of the induced map it is enough to show, under the hypothesis $\alpha\leq \mcal{L}\text{-}\dim \Gamma_\infty^+$, that for all $a\leq b\in\Gamma^+_\infty$, $a\sim_\alpha b$ if and only if $b-a\in C_\alpha$ or $a=b=\infty$.

Suppose that $a\in\Gamma^+$ and $b=\infty$. If $a\sim_\alpha b=\infty$ then $[a,b]=[a,\infty]\cong \Gamma_\infty^+$. But then $\mcal{L}\text{-}\dim\Gamma_\infty^+<\alpha$ which contradicts our hypothesis. Therefore if $b=\infty$ and $a\sim_\alpha b$ then $a=\infty$.


We now prove by ordinal induction that for all $a\leq b\in\Gamma^+$, $a\sim_\alpha b$ if and only if $b-a\in C_\alpha$. Since $[a,b]\cong [0,b-a]$ for all $a\leq b\in\Gamma^+$, it's enough to show that for all $c\in\Gamma^+$, $0\sim_\alpha c$ if and only if $c\in C_\alpha$. The base case, $\alpha=0$, is trivial. The limit step follows directly from the definition of $\sim_\alpha$ and $C_\alpha$. Suppose $\alpha=\beta+1$ and that we have proved the statement for $\beta$. Now $0\sim_{\beta+1}c$ if and only if $\mcal{L}\text{-}\dim[0/\sim_{\beta},c/\sim_{\beta}]=0$. By the induction hypothesis, $[0/\sim_{\beta},c/\sim_{\beta}]\cong [0+C_\beta,c+C_\beta]\subseteq (\Gamma/C_\beta)_\infty^+$. By \ref{ldimgroupbasecase}, $\mcal{L}\text{-}\dim[0+C_\beta,c+C_\beta]=0$ if and only if $c+C_\beta\in C_\mcal{L}(\Gamma/C_\beta)$. Hence $0\sim_{\beta+1} c$ if and only if $c\in C_{\beta+1}$.

The statement $(i)$ is a direct consequence of the main statement. For $(ii)$, first suppose $\Gamma_\infty^+$ has $\mcal{L}$-dimension $\alpha$. Then $\mcal{L}\text{-}\dim(\Gamma/C_{\mcal{L},\alpha})_\infty^+=0$. Thus, for all $c\in \Gamma^+$, $c+C_{\mcal{L},\alpha}\sim_\mcal{L} 0+C_{\mcal{L},\alpha}$. Hence $c\in C_{\mcal{L},\alpha+1}$. Therefore $C_{\mcal{L},\alpha+1}=\Gamma$. For the converse, suppose $\mcal{L}\text{-}\dim\Gamma_\infty^+=\infty$. Then $(\Gamma/C_{\mcal{L},\alpha})_\infty^+\cong \Gamma_\infty^+/\sim_{\mcal{L},\alpha}$ for all $\alpha\in\Ord$. Hence, if $C_{\mcal{L},\alpha}=\Gamma$ then $\Gamma_\infty^+/\sim_{\mcal{L},\alpha}\in\two$ which contradicts our assumption. Thus  $C_{\mcal{L},\alpha}\neq\Gamma$ for all $\alpha\in\Ord$.
%
%
\end{proof}

\begin{cor}\label{dimforgroup2}
The m-dimension of $\Gamma_\infty^+$ is the first ordinal $\alpha$ such that $C_{\two,\alpha}=\Gamma$ and the breadth of $\Gamma_\infty^+$ is the first ordinal $\alpha$ such that $\Gamma/C_{\mathbbm{Ch},\alpha}$ is totally ordered.
%
\end{cor}
\begin{proof} This follows from \ref{dimforgroup1} and \ref{Ldimtrivial}.
\end{proof}

%
%
%

\section{Superdecomposable Pure-injective modules}\label{SSuperDecPIM}
\noindent
Throughout, unless otherwise stated, let $R$ be a B\'ezout domain. Recall that there is a correspondence between the saturated multiplicatively closed subsets of a B\'ezout and the convex $\ell$-subgroups of its value group. We start by defining, \ref{SandT}, the saturated multiplicatively closed subsets of $R$ corresponding to the convex $\ell$-subgroups $C_{\two,\alpha}$ and $C_{\Ch,\alpha}$ of $\Gamma(R)$ for ordinals $\alpha$.


\begin{remark}\label{irredlatvallat}
Let $a\in R\backslash\{0\}$. Then
\begin{itemize}
\item  $\Gamma(R)\supseteq [0,aR]\in \two$ if and only if $a$ is either irreducible or a unit; and
\item  $\Gamma(R)\supseteq [0,aR]\in\Ch$ if and only if $R/aR$ is a valuation ring.
\end{itemize}
\end{remark}


In the next definition we identify all localisations of $R$ with subsets of its field of fractions.
\begin{definition}\label{SandT}
Define $S(R)$ to be the multiplicatively closed set generated by the units and all irreducible elements of $R$. Define $T(R)$ to be the multiplicatively closed subset generated by the units of $R$ and all $a\in R$ such that $R/aR$ is a valuation ring.   For each ordinal $\alpha$, we define multiplicatively closed sets $S_\alpha(R)$ and $T_\alpha(R)$ by ordinal induction. Let $S_0(R)=T_0(R)$ be the set of units of $R$. Let $S_1(R):=S(R)$ and $T_1(R):=T(R)$. Define \[S_{\alpha+1}(R):=R\cap S(R_{S_{\alpha}(R)}) \ \ \ \text{and}\ \ \ T_{\alpha+1}(R):=R\cap T(R_{T_{\alpha}(R)}).\]
For $\lambda$ a limit ordinal, define
\[S_{\lambda}(R):=\bigcup_{\alpha<\lambda} S_\alpha(R) \ \ \ \text{and}\ \ \ T_{\lambda}(R):=\bigcup_{\alpha<\lambda}T_{\alpha}(R).\]
Finally, let
\[
S_\infty(R):=\bigcup_{\alpha\in\Ord}S_\alpha(R)\ \  \text{ and } \  \ T_\infty (R):=\bigcup_{\alpha\in\Ord} T_\alpha(R).
\]
\end{definition}



\begin{remark}\label{correspmultgroup}
For each ordinal $\alpha$, $S_\alpha(R)$ and $T_\alpha(R)$ are saturated multiplicatively closed subsets. Moreover,
under the correspondence between saturated multiplicatively closed subsets of $R$ and convex $\ell$-subgroups of $\Gamma(R)$ (see the last paragraph of section \ref{SPreI}),
\begin{enumerate}[(i)]
\item $S_\alpha(R)$ corresponds to $C_{\two,\alpha}$, and
\item $T_\alpha(R)$ corresponds to $C_{\Ch,\alpha}$.
\end{enumerate}
%
\end{remark}
\begin{proof}We show that for all $a\in R$,
\begin{enumerate}[(i)]
\item $a\in S_\alpha(R)$ if and only if $aR \in C_{\two, \alpha}$, and
\item $a\in T_\alpha(R)$ if and only if  $aR\in C_{\Ch,\alpha}$.
\end{enumerate}
This proves the second statement of the remark and implies that $S_\alpha(R)$ and $T_\alpha(R)$ are saturated.
It follows from \ref{Ldimlgroup} and \ref{irredlatvallat} that $a\in S(R)$ if and only if $aR\in C_\two$, and, $a\in T(R)$ if and only if $aR\in C_{\Ch}$. So the case $\alpha=1$ is true. Since $C_{\two,0}=C_{\Ch,0}=\{0\}$ and $S_0(R)=T_0(R)$ is the set of units of $R$, the case $\alpha=0$ is true. We now prove $(i)$ for all $\alpha$ by ordinal induction (the proof of $(ii)$ is identical).
The limit case follows from the definitions of $C_{\two,\lambda}$ and $S_{\lambda}(R)$ when $\lambda$ is a limit ordinal.
Suppose we have proved the claim for $\alpha$. The the map $\Gamma(R)\rightarrow \Gamma(R_{S_\alpha(R)})$ given by $aR\mapsto aR_{S_\alpha(R)}$ for $a\in Q^\times$ is surjective with kernel $C_{\two,\alpha}(\Gamma(R))$. So, by definition of $C_{\two,\alpha+1}(\Gamma(R))$,  $aR\in C_{\two,\alpha+1}(\Gamma(R))$ if and only if $aR_{S_\alpha(R)}\in C_\two(\Gamma(R_{S_\alpha(R)}))$. By the $\alpha=1$ case, this is true if and only if $a\in S(R_{S_\alpha(R)})$. Hence, by definition of $S_{\alpha+1}(R)$, if and only if $a\in S_{\alpha+1}(R)$ as required.
%
%
%
%
%
%
%
%
%
%
%
\end{proof}

\begin{lemma}\label{dimintermsofmult}
\noindent
\begin{enumerate}
\item For each ordinal $\alpha\leq m\text{-}dim \, \Gamma(R)_\infty^+$, \[\Gamma(R)_\infty^+/\sim_{\two,\alpha}\text{ is isomorphic to }\Gamma(R_{S_\alpha(R)})_\infty^+.\] In particular, the m-dimension of $\Gamma(R)_\infty^+$ is equal to the least ordinal $\alpha$ such that $R_{S_\alpha(R)}$ is a field (equivalently $S_\alpha(R)=R\backslash\{0\}$) and if no such $\alpha$ exists then the m-dimension of $\Gamma(R)_\infty^+$ is undefined.
\item For each $\alpha\leq  br\Gamma(R)_{\infty}^+$, \[\Gamma(R)_\infty^+/\sim_{\mathbbm{Ch},\alpha}\text{ is isomorphic to }\Gamma(R_{T_\alpha(R)})_\infty^+.\]In particular, the breadth of $\Gamma(R)^+_\infty$ is equal to the least ordinal $\alpha$ such that $R_{T_\alpha(R)}$ is a valuation ring and if no such $\alpha$ exists then the breadth of $\Gamma(R)_\infty^+$ is undefined.

\end{enumerate}
\end{lemma}
\begin{proof}
This follows from \ref{dimforgroup1}, \ref{dimforgroup2} and  \ref{correspmultgroup}.
%
\end{proof}

\begin{theorem}\label{superdecvring}\cite[proof of 12.12]{Serialrings}
For $R$ a valuation ring, the following are equivalent.
\begin{enumerate}
\item There exists a superdecomposable pure-injective $R$-module.
\item The lattice of principal ideals of $R$ has a dense suborder.
\end{enumerate}
\end{theorem}

\begin{lemma}\label{widthnotmdimvalgroup}
Suppose that $\Gamma(R)_\infty^+$ has breadth but $\Gamma(R)_\infty^+$ does not have m-dimension. There exists $a\in R_{S_\infty(R)}$ such that $R_{S_\infty(R)}/aR_{S_\infty(R)}$ is a valuation ring whose lattice of principal ideals is dense. In particular, $R_{S_\infty(R)}$ has a superdecomposable pure-injective module.
\end{lemma}
\begin{proof}
Since $\Gamma(R)_\infty^+$ does not have m-dimension, $\Gamma(R_{S_\infty(R)})$ is a non-trivial $\ell$-group such that all proper intervals are dense. Since $\Gamma(R)_\infty^+$ has breadth and $\Gamma(R_{S_\infty(R)})_\infty^+$ is a quotient of $\Gamma(R)_\infty^+$, $\Gamma(R_{S_\infty(R)})_\infty^+$ also has breadth. Therefore there exists $a\in R_{S_\infty(R)}\backslash \{0\}$ such that the interval $[0,aR_{S_\infty(R)}]$ in $\Gamma(R_{S_\infty(R)})_\infty^+$ is proper and totally ordered. Therefore $R_{S_\infty(R)}/aR_{S_\infty(R)}$ is a valuation ring whose lattice of principal ideals is dense. So, by \ref{superdecvring}, $R_{S_\infty(R)}/aR_{S_\infty(R)}$ has a superdecomposable pure-injective module.
%
%
\end{proof}


\begin{lemma}\label{Bezpp1}
Let $R$ be a Pr\"ufer domain. Suppose that for all proper finitely generated ideals $A\lhd R$, $R/A$ is not a valuation ring. The pp-type of $1\in R$ is superdecomposable.
\end{lemma}
\begin{proof}
Suppose that $A\lhd R$ is a proper finitely generated ideal. Since $R/A$ is not a valuation ring, there exist finitely generated ideals $B',C'\supseteq A$ which are
incomparable by inclusion. Let $B=(B':B'+C')$ and $C=(C':B'+C')$. Then $B,C\supseteq A$ and, by \ref{arithdef}, $B+C=R$. Moreover, $B$ and $C$  are incomparable since $B'$ and $C'$ are incomparable.

Let $p$ be the pp-type of $1\in R$. Suppose $A,E\lhd R$ are finitely generated ideals and $A|x+xE=0\notin p$. Then $A\neq R$ and $E$ is non-zero. Take $B,C\lhd R$ as in the previous paragraph. Then $A|x\leq B|x$ and $A|x\leq C|x$ and, by \ref{commringsumintpp}, $B|x+C|x$ is equivalent to $R|x$.
Thus $B|x+xE=0\geq A|x+xE=0$ and $C|x+xE=0\geq A|x+xE=0$, and $B|x+xE=0+C|x+xE=0\in p$. Since $B,C\neq R$, $B|x+xE=0,C|x+xE=0\notin p$. So, since all $\psi\in\pp_R^1$ are of the form $\bigwedge_{i=1}^n(A_i|x+xE_i=0)$, $p$ is superdecomposable.
\end{proof}

\begin{proposition}\label{widthvalimp1superdec}
Suppose that $\Gamma(R)^+_\infty$ does not have breadth. Then the pp-type of $1$ in $R_{T_\infty(R)}$ is superdecomposable. In particular, there exists a superdecomposable pure-injective $R$-module.
\end{proposition}
\begin{proof}
Since $\Gamma(R)^+_\infty$ does not have breadth, for all $a\in \Gamma(R_{T_\infty(R)})_\infty^+$, if $[0,a]$ is a chain then $a=0$. Therefore, for all $a\in R_{T_\infty(R)}$, if $R_{T_\infty(R)}/aR_{T_\infty(R)}$ is a valuation ring then $a$ is a unit in $R_{T_\infty(R)}$. So the result follows from \ref{Bezpp1}.
\end{proof}

\begin{theorem}\label{vgroupmdimsuperdec}
Let $S$ be a Pr\"ufer domain.  If $\Gamma(S)^+_\infty$ does not have m-dimension then there is a superdecomposable pure-injective $S$-module.
\end{theorem}
\begin{proof}
Suppose that $R$ is a B\'ezout domain. If $\Gamma(R)^+_\infty$ does not have breadth then $R$ has a superdecomposable pure-injective module by \ref{widthvalimp1superdec}. So, suppose that $\Gamma(R)^+_\infty$ has breadth but not m-dimension, then $R$ has a superdecomposable pure-injective module by \ref{widthnotmdimvalgroup}.

Now suppose $S$ is a Pr\"ufer domain. Take $R$ a B\'ezout domain such that $\Gamma(R)\cong \Gamma(S)$. By \ref{BeztoPruf}, $\pp_S^1$ is isomorphic to $\pp_R^1$ as a lattice. If $\Gamma(S)_\infty^+$ does not have m-dimension then neither does $\Gamma(R)_\infty^+$. By what we have just proved, there is a superdecomposable pure-injective $R$-module. Therefore, by \ref{sdtypetopi}, there exists a superdecomposable pp-type in $\pp_R^1$. So, since $\pp_S^1$ is isomorphic to $\pp_R^1$, $\pp_S^1$ has a superdecomposble pp-type. Hence there is a superdecomposable pure-injective $S$-module.
\end{proof}

\section{Calculating the breadth}\label{SCalcBreadth} Throughout, unless otherwise stated, $R$ is a B\'ezout domain.
The main purpose of this section is to show that the breadth of $\pp_R^1$ is equal to the m-dimension of $\Gamma(R)_\infty^+$ when $R$ is a Pr\"ufer domain. The main work is in
Proposition \ref{breadthcollapseloc}, which implies the statement for B\'ezout domains. We then use our transfer theorem \ref{BeztoPruf} to extend the result to Pr\"ufer domains.


\begin{proposition}\label{breadthcollapseloc}
Let $R$ be a B\'ezout domain which is not a field and let $S:=S(R)$. The surjective lattice homomorphism
\[\pi_S:\pp_R^1\rightarrow \pp_{R_S}^1\] induced by restriction of scalars along the epimorphism $R\rightarrow R_S$ is such that for all $\phi,\psi\in\pp_R^1$, $\pi_S(\phi)=\pi_S(\psi)$ if and only if the interval $[\phi\wedge\psi,\phi+\psi]$ has breadth $0$.
\end{proposition}

In order to prove the proposition we need to show:
\begin{enumerate}
\item If $[\psi,\phi]$ is totally ordered then $[\pi_S(\psi),\pi_S(\phi)]$ is the trivial interval.
\item If $\psi\leq \phi$ and $\pi_S(\phi)=\pi_S(\psi)$ then $[\psi,\phi]$ has has breadth $0$.
\end{enumerate}

\begin{remark}\label{chainsumint}
Let $L$ be a modular lattice, $a_1,\ldots,a_n\in L$ and $b_1,\ldots, b_m\in L$. If $[(\bigsqcup_{i=1}^na_i)\sqcap\bigsqcap_{j=1}^mb_j,\bigsqcup_{i=1}^n a_i ]$ is a chain then $[b_j\sqcap a_i,a_i]$ is a chain for all $1\leq i\leq n$ and $1\leq j\leq m$.
\end{remark}

\begin{lemma}\label{trivialorchain}
Suppose $R$ is not a field. Let $a,d,g,h,w\in R$ with $g,h\neq 0$ be such that $aR+dR=R$. If the interval \[[(awg|x+xh=0)\wedge g|x\wedge xdwh=0, g|x\wedge xdwh=0],\] is trivial then $w$ is a unit and if it is a chain then either $w$ is a unit of $R$ or $w$ is irreducible.
\end{lemma}
\begin{proof}
The first statement follows from \ref{Zgbasicmanip} because $(dwhR:h)+(awgR:g)=w(dR+aR)$.

The condition $aR+dR=R$ implies that either $a\neq 0$ or $d\neq 0$. We will prove the second claim of the lemma under the assumption that $a\neq 0$ (Prest's duality interchanges this case with the case where $d\neq 0$).

Suppose that $w$ is not a unit and $w$ is not irreducible. There exists non-units $u,v\in R$ with $v\neq 0$ such that $w=uv$. Let $\phi$ be $g|x\wedge xduvh=0$ and $\psi$ be $auvg|x+xh=0$. Let $\phi_1$ be $avg|x+xh=0$ and let $\phi_2$ be $auvg|x+xvh=0$.

Note that $\phi_1,\phi_2\geq \psi$. Hence it is enough to show that $\phi_1\wedge\phi$ and $\phi_2\wedge\phi$ are incomparable elements of $\pp^1_R$.
Using the distributivity of $\pp_R^1$, it is easy to see that $\phi_1\wedge \phi$ is
\[avg|x\wedge xduvh=0+ g|x\wedge xh=0\] and $\phi_2\wedge \phi$ is
\[auvg|x\wedge xduvh=0 +g|x\wedge xvh=0.\]

\noindent
\textbf{Claim 1:} $avg|x\wedge xduvh=0\nleq \phi_2$

\noindent
Let $I:=avgduvhR$ and $M:=R/I$. In $M$, $avg+I$ satisfies $avg|x\wedge xduvh=0$. Now $\phi_2(M)$ is equal to
\[(avgR\cap (I:duvh)+gR\cap(I:vh))+I=(auvgR+avgduR)+I=auvgR+I.\] Since $u$ is not a unit and $auvgR\supseteq I$, $avg+I\notin auvgR+I=\phi_2(M)$.

\smallskip
\noindent
\textbf{Claim 2:} $g|x\wedge xvh=0\nleq\phi_1$

\noindent
Let $I:=gvhR$ and $M:=R/I$. In $M$, $g+I$ satisfies $g|x\wedge xvh=0$. Now $\phi_1(M)$ is equal to
\[(avgR\cap (gvhR:duvh)+gvR)+I=gvR+I.\] Since $v$ is not a unit and $gvR\supseteq I$, $g +I\notin gvR+I=\phi_1(M)$.

\smallskip
\noindent
By claim $1$, $\phi_1\wedge\phi\nleq \phi_2$ and by claim $2$, $\phi_2\wedge\phi\nleq \phi_1$.
\end{proof}


\begin{lemma}
Suppose $R$ is not a field. If $[\psi\wedge \phi,\phi]$ is a chain then $[\pi_S(\psi\wedge \phi),\pi_S(\phi)]$ is the trivial interval.
\end{lemma}
\begin{proof}
By \ref{Specialform}, we may assume $\phi$ is of the form $\sum_{i=1}^n g_i|x\wedge xf_i=0$ and $\psi$ is of the form $\bigwedge_{j=1}^n c_j|x+xb_j=0$. Now if  $[\pi_S(\psi_j\wedge \phi_i),\pi_S(\phi_i)]$ is a trivial interval for all $1\leq i\leq n$ and $1\leq j\leq m$ then so is $[\pi_S(\psi\wedge \phi),\pi_S(\phi)]$. So, by the \ref{chainsumint}, it is enough to prove the statement of the lemma for $\phi$ of the form $g|x\wedge xf=0$ and $\psi$ of the form $c|x+xb=0$.

Let $a',h,d'\in R$ be such that $a'gR=cR\cap gR$, $bR+fR=hR$ and $f=d'h$. Then $\psi\wedge \phi$ is equivalent to
\[(a'g|x+xh=0)\wedge g|x\wedge xf=0\] because $c|x\wedge g|x$ is equivalent to $a'g|x$ and $xb=0\wedge xf=0$ is equivalent to $xh=0$. Since taking intersections and sums of ideals commutes with localising, this equivalence also holds for $\phi$ and $\psi$ viewed as formulae over $R_S$. So, we may assume, $\phi$ is $g|x\wedge xd'h=0$ and $\psi$ is $a'g|x+xh=0$.

Let $w,a,d\in R$ be such that $wR=a'R+d'R$, $a'=aw$ and $d'=dw$. If $w\neq 0$ this implies $aR+dR=R$. If $w=0$ then we may choose $a=d=1$, so $aR+dR=R$. If $g=0$ then $\phi$ is equivalent to $x=0$ and if $h=0$ then $\psi$ is equivalent to $x=x$. Thus there is no harm in assuming that $g,h\neq 0$. We are now exactly in the situation of \ref{trivialorchain}.
Therefore, if $[\psi\wedge \phi,\phi]$ is a chain then either $w$ is a unit or $w$ is irreducible. But then $w$ is a unit in $R_S$. Therefore $(d'hR:h)+(a'gR:g)=R$. Hence, by \ref{Zgbasicmanip}, $[\pi_S(\psi\wedge \phi),\pi_S(\phi)]$ is the trivial interval.
%
%
%
%
%
%
\end{proof}


It is easy to see that a non-zero element $w$ of a B\'ezout domain $R$ is irreducible if and only if $wR$ is maximal.

\begin{lemma}\label{xxcxchain}
Let $c\in R$ be irreducible and let $\mfrak{m}:=cR$. Then $[c|x,x=x]$ is isomorphic to $\Gamma(R_{\mfrak{m}})^+_\infty$ and $[x=0,xc=0]$ is isomorphic to $(\Gamma(R_{\mfrak{m}})^+_\infty)^{op}$.
\end{lemma}
\begin{proof}
Suppose $\phi\geq c|x$. Let $d_i,b_i\in R$ be such that $\phi$ is equivalent to $\bigwedge_{i=1}^n(d_i|x+xb_i=0)$. Since $\phi\geq c|x$ and $R$ is an integral domain, $\bigcap_{i=1}^nd_iR=\phi(R)\supseteq cR$. So $c\in d_iR$ for all $1\leq i\leq n$. Therefore, for all $1\leq i\leq n$, either $d_i$ is a unit and hence $d_i|x$ is equivalent to $x=x$, or, $d_iR=cR$ and hence $d_i|x$ is equivalent to $c|x$. So we may assume $\phi$ is equivalent to  $\bigwedge_{i=1}^n(c|x+xb'_i=0)$ for some $b_i'\in R$. Since $\pp_R^1$ is distributive, $\bigwedge_{i=1}^n(c|x+xb'_i=0)$ is equivalent to $c|x+\bigwedge_{i=1}^nxb'_i=0$. Let $b\in R$ be such that $bR=\sum_{i=1}^nb_i'R$. By \ref{commringsumintpp}, $\bigwedge_{i=1}^nxb'_i=0$ is equivalent to $xb=0$. Therefore $\phi$ is equivalent to $c|x+xb=0$.

Now we show that $c|x+xb=0\leq c|x+xb'=0$ if and only if $b'\in bR_{\mfrak{m}}$. The formula $xb=0$ is freely realised by $1+bR\in R/bR$. Therefore $xb=0\leq c|x+xb'=0$ if and only if $1\in cR+(bR:b')$. Now, $1\in cR+(bR:b')$ if and only if there exist $r,s,\mu\in R$ such that $1=cr+\mu$ and $\mu b'=bs$. So $1\in cR+(bR:b')$ implies $\mu\notin \mfrak{m}$ and hence $b'\in bR_{\mfrak{m}}$. Conversely, if $b'\in bR_{\mfrak{m}}$ then there exists $s,\mu'\in R$ such that $\mu' b'=bs$ and $\mu'\notin\mfrak{m}$. Therefore there exist $r,t\in R$ such that $1=cr+\mu't$. Since $1=cr+\mu't$ and $(\mu' t) b'=b(st)$, we get $1\in cR+(bR:b')$. Thus, we have shown that $xb=0\leq c|x+xb'=0$ if and only if $b'\in bR_{\mfrak{m}}$.

Therefore, the map which sends $bR_{\mfrak{m}}$ to $c|x+xb=0$ is a lattice isomorphism from $\Gamma(R_{\mfrak{m}})^+_\infty$ to $[c|x,x=x]$. The second claim of the lemma follows from the first using Prest's duality.
\end{proof}

\begin{proposition}\label{direction2} 
Let $\phi\geq \psi\in\pp_R^1$. If $\pi_S(\phi)=\pi_S(\psi)$ then $[\psi,\phi]$ has breadth $0$.
\end{proposition}
\begin{proof}
Let $\mcal{C}_\Ch$ be the closed subset of $\Zg(R)$ consisting of exactly those indecomposable pure-injective $R$-modules $N$ such that if $[\psi,\phi]\subseteq \pp_R^1$ is totally ordered then $\phi(N)=\psi(N)$. By \ref{closedsetforLdim}, for $\phi\geq \psi$, $[\psi,\phi]$ has breadth $0$ if and only if $\phi(N)=\psi(N)$ for all $N\in\mcal{C}_\Ch$. Therefore, the proposition is equivalent to $\Zg(R_S)\supseteq \mcal{C}_\Ch$.

Suppose $N\in\Zg(R)\backslash\Zg(R_S)$. Then some element of $S$ acts non-bijectively on $N$. Hence some irreducible element $c\in R$ acts non-bijectively on $N$. So either $N\in\left(\nf{x=x}{c|x}\right)$ or $N\in \left(\nf{xc=0}{x=0}\right)$. By \ref{xxcxchain}, $[c|x,x=x]$ and $[x=0,xc=0]$ are totally ordered. Therefore $N\notin \mcal{C}.$
\end{proof}

Thus we have now proved \ref{breadthcollapseloc}. We record a corollary which will be used in the next section.

\begin{cor}\label{movetogoodopen}
Let $N\in\Zg(R)$.
There exist $\phi\geq \psi\in\pp_R^1$ such that $[\psi,\phi]$ is totally ordered with $N\in\left(\nf{\phi}{\psi}\right)$ if and only if there exists $c\in R$ irreducible such that $N\in\left(\nf{x=x}{c|x}\right)\cup\left(\nf{xc=0}{x=0}\right)$.
\end{cor}
\begin{proof}
%
%
Recall that $[\psi,\phi]$ has breadth $0$ if and only if there exist $\psi=:\phi_0\leq \phi_1\leq\ldots\leq \phi_{n+1}:=\phi$ such that $[\phi_i,\phi_{i+1}]$ is totally ordered for $0\leq i\leq n$. Therefore, it is enough to prove the corollary with ``totally ordered'' replaced with ``breadth $0$''. Proposition \ref{breadthcollapseloc} implies that for all $\phi\geq \psi\in \pp_R^1$, $[\psi,\phi]$ has breadth $0$ if and only if $\left(\nf{\phi}{\psi}\right)\cap \Zg(R_S)=\emptyset$. Hence, since $\Zg(R_S)$ is closed, here exist $\phi\geq \psi\in\pp_R^1$ such that $[\psi,\phi]$ has breadth $0$ with $N\in\left(\nf{\phi}{\psi}\right)$ if and only if $N\notin \Zg(R_S)$. Since $S$ is generated by the irreducible elements of $R$, $N\notin \Zg(R_S)$ if and only if there exists an irreducible element $c\in R$ such that $N\in\left(\nf{x=x}{c|x}\right)$ or $N\in \left(\nf{xc=0}{x=0}\right)$.
%
%
%
%
%
%
%
%
%
%
%
%
%
%
\end{proof}

\begin{theorem}\label{calcbreadth}
Let $R$ be a Pr\"ufer domain. The breadth of $\pp_R^1$ is equal to the m-dimension of $\Gamma(R)_\infty^+$. In particular, $\pp_R^1$ has breadth if and only if $\Gamma(R)_\infty^+$ has m-dimension.
\end{theorem}
\begin{proof}By \ref{BeztoPruf}, it is enough to prove the theorem for $R$ a B\'ezout domain. For each $\alpha\in\Ord$, let $\pi_{S_\alpha(R)}:\pp_R^1\rightarrow \pp_{R_{S_\alpha(R)}}$ be the surjective lattice homomorphism induced by restriction of scalars along the epimorphism $R\rightarrow R_{S_\alpha(R)}$. We prove by ordinal induction that, whenever $\alpha\leq \mdim \Gamma(R)_\infty^+$, the congruence kernel of $\pi_{S_\alpha(R)}$ is $\sim_{\Ch,\alpha}$.
The base case $\alpha=0$ is trivial. Suppose $\alpha=\beta+1$, the statement holds for $\beta$ and $\alpha\leq \mdim \Gamma(R)_\infty^+$. Using \ref{dimintermsofmult}, this implies that $1\leq \mdim \Gamma(R_{S_\beta(R)})_\infty^+$ and hence $R_{S_\beta(R)}$ is not a field. Thus, by \ref{breadthcollapseloc}, the congruence kernel of $\pi_{S_1(R_{S_\beta(R)})}:\pp_{R_{S_\beta(R)}}\rightarrow \pp_{R_{S_{\beta+1}(R)}}$ is $\sim_\Ch$. By the induction hypothesis, $\pi_{S_\beta(R)}$ has congruence kernel $\sim_{\Ch,\beta}$. By definition of $S_{\beta+1}(R)$, $\pi_{S_{\beta+1}(R)}=\pi_{S_1(R_{S_\beta(R)})}\circ\pi_{S_{\beta}(R)}$. Hence $\pi_{S_{\beta+1}(R)}$ has congruence kernel $\sim_{\Ch,\beta+1}$.

Suppose $\lambda$ is a limit ordinal, the statement holds for all $\alpha<\lambda$ and that $\lambda\leq \mdim \Gamma(R)_\infty^+$. Since we can axiomatise the class of $R_{S_\alpha(R)}$-modules for each ordinal $\alpha$ in the language of $R$-modules, compactness\footnote{With a bit more work this argument can be replaced with one using compactness of the Ziegler spectrum.} of first order logic implies that for $\phi,\psi\in \pp_R^1$, $\pi_{S_\lambda(R)}(\phi)=\pi_{S_\lambda(R)}(\psi)$ if and only if $\pi_{S_\alpha(R)}(\phi)=\pi_{S_\alpha(R)}(\psi)$ for some $\alpha<\lambda$.
By definition of $\sim_{\Ch,\lambda}$ at limit steps, it now follows that $\pi_{S_\lambda(R)}$ has congruence kernel $\sim_{\Ch,\lambda}$.

We have shown that  $\pp_R^1/\sim_{\Ch,\alpha}\cong \pp^1_{R_{S_\alpha(R)}}$ for all ordinals $\alpha\leq \text{m-dim}\Gamma(R)_\infty^+$. Recall, \ref{dimintermsofmult}, that if $\alpha=\text{m-dim}\Gamma(R)_\infty^+$ then $R_{S_\alpha(R)}$ is a field. So, in particular, $\pp^1_{R_{S_\alpha(R)}}$ has breadth $0$. Hence $\br \pp_R^1\leq \alpha$.
Now suppose for a contradiction that $\br \pp_R^1=\beta<\text{m-dim}\Gamma(R)_\infty^+$. Then $\beta+1\leq \text{m-dim}\Gamma(R)_\infty^+$. Hence $\pp_R^1/\sim_{\Ch,\beta+1}\cong \pp^1_{R_{S_{\beta+1}(R)}}$. This gives a contradiction since $\pp^1_{R_{S_{\beta+1}(R)}}$ has at least $2$ elements.
\end{proof}

In the next section we will investigate the relationship between the value of the m-dimension of $\pp_R^1$
and the m-dimension of $\Gamma(R)_\infty^+$.
What we have proved in this section is already enough to show, \ref{weakmdim}, that these two dimensions coexist.

Recall that every arithmetical ring $R$ satisfies the isolation condition, see \ref{CBmdimdist}.
This implies that, \cite[10.19]{PrestBluebook}, for all $N\in \Zg(R)$,
\[\CBrank N=\inf\{\mdim [\psi,\phi] \st N\in\left(\phi/\psi\right)\}.\]

\begin{remark}\label{GammappRmdimLB}
Let $R$ be a Pr\"ufer domain. Then $\mdim \Gamma(R)_\infty^+\leq \mdim \pp_R^1$.
\end{remark}
\begin{proof} We argue that $\Gamma(R)_\infty^+$ is a quotient of $\pp_R^1$. Then, by \cite[10.4]{PrestBluebook}, $\mdim \Gamma(R)_\infty^+\leq \mdim\pp_R^1$.
For any coherent ring, so in particular for Pr\"ufer domains, the pp-definable subsets of $R$ as an $R$-module are exactly the finitely generated ideals of $R$ by \cite[14.16]{PrestBluebook}. Thus, if we identify $\infty\in \Gamma(R)_\infty^+$ with the ideal $\{0\}\lhd R$, then the map $\phi\in\pp_R^1\mapsto D\phi({}_RR)$, where $D$ denotes Prest's duality, is a surjective lattice homomorphism.
\end{proof}

\begin{theorem}\label{weakmdim}
Let $R$ be a Pr\"ufer domain. The lattice $\pp_R^1$ has m-dimension if and only if $\Gamma(R)_\infty^+$ has m-dimension.
\end{theorem}
\begin{proof} By \ref{BeztoPruf}, it is enough to prove the theorem for B\'ezout domains. The reverse direction follows from \ref{GammappRmdimLB}.
For the other direction, we show that if $\Gamma(R)_\infty^+$ has m-dimension then $\Zg(R)$ has CB rank. Suppose $\Gamma(R)_\infty^+$ has m-dimension. It is enough to show that for each ordinal $\alpha$, $\Zg(R)^{(\alpha)}$ is either empty or contains a point with CB rank. We set $S_\alpha:=S_\alpha(R)$. For each $\alpha\in\Ord$, let $\beta(\alpha)$ be the supremum of those $\gamma\leq \mdim \Gamma(R)_\infty^+$ such that
\[\Zg(R)^{(\alpha)}\subseteq \Zg(R_{S_{\gamma}}).\] Then
\[\Zg(R)^{(\alpha)}\subseteq \Zg(R_{S_{\beta(\alpha)}}).\] Moreover, for all $\alpha\in\Ord$, either $\beta(\alpha)=\mdim \Gamma(R)_\infty^+$ or there exists $N\in \Zg(R)^{(\alpha)}$ which is not in $\Zg(R_{S_{\beta(\alpha)+1}})$. In the first case, $\Zg(R_{S_{\beta(\alpha)}})=\{Q\}$ where $Q$ denotes the field of fractions of $R$. So, in particular, $\Zg(R)^{(\alpha)}$ has CB rank.

For the second case, it is enough to show that, for every $\beta\leq \mdim \Gamma(R)_\infty^+$, every point in $\Zg(R_{S_\beta})\backslash \Zg(R_{S_{\beta+1}})$ has CB rank in $\Zg(R_{S_\beta})$ because if a point has CB rank in a topological space then it has CB rank in every subspace containing it. Replacing $R$ be $R_{S_\beta}$ it's enough to show that every point in $\Zg(R)\backslash \Zg(R_S)$ has CB rank. If $N\in \Zg(R)\backslash \Zg(R_S)$ then some element of $S$ acts non-invertibly on $N$ and hence some irreducible element $c\in R$ acts non-invertibly on $N$. Therefore $N\in \left(\nf{x=x}{c|x}\right)$ or $N\in\left(\nf{xc=0}{x=0}\right)$.
By \ref{xxcxchain}, $[c|x,x=x]$ is isomorphic to $\Gamma(R_{\mfrak{m}})^+_\infty$ and $[x=0,xc=0]$ is isomorphic to $(\Gamma(R_{\mfrak{m}})^+_\infty)^{op}$ where $\mfrak{m}:=cR$. Since $\Gamma(R)_\infty^+$ has m-dimension and $\Gamma(R_\mfrak{m})_\infty^+$ is a quotient of it, $\Gamma(R_\mfrak{m})_\infty^+$ also has m-dimension. Therefore $[c|x,x=x]$ and $[x=0,xc=0]$ both have m-dimension. So by the comment just before this theorem, $N$ has CB rank as required.
\end{proof}


The next theorem, which tells us where the superdecomposable pure-injective modules live, is a consequence of \ref{xxcxchain}.

\begin{theorem}\label{locsuperdecomp}
Let $R$ be a B\'ezout domain. All superdecomposable pure-injective $R$-modules are $R_{S_\infty(R)}$-modules.
\end{theorem}
\begin{proof}
Suppose an $R$-module $M$ is not an $R_{S_\infty(R)}$-module. Then $M$ is not an $R_{S_\alpha(R)}$-module for some ordinal $\alpha$. Let $\beta$ be the least such ordinal. It is easy to see that $\beta$ must be a successor ordinal. So let $\alpha:=\beta-1$. By replacing $R$ be $R_{S_\alpha(R)}$, we may reduce to the case where $M$ is an $R$-module but not an $R_{S(R)}$-module. Therefore, there exists some element of $S(R)$ which does not act invertibly on $M$ and hence there exists some $c\in R$ irreducible such that $c$ does not act invertibly on $M$. Therefore $M\neq (c|x)(M)$ or $(xc=0)(M)\neq 0$. By \ref{xxcxchain}, the intervals $[x=x,c|x]$ and $[xc=0,x=0]$ are both totally ordered. Hence, by \cite[7.3.2]{PSL}, $M$ is not superdecomposable.
\end{proof}

\section{Calculating the Cantor-Bendixson rank of $\Zg(R)$}\label{SCBrank}

Throughout this section, unless otherwise stated, $R$ is a B\'ezout domain. In this section we show, \ref{upperbound}, that if $\Gamma(R)_{\infty}^+$ has m-dimension $\alpha$ then the CB rank of $\Zg(R)$ is bounded above by $\alpha\cdot 2$ and below by $\alpha$. Reinterpreting a result of Puninski will show that if $R$ is a valuation domain and $\Gamma(R)_\infty^+$ has m-dimension $\alpha$ then $\Zg(R)$ has CB rank $\alpha\cdot 2$. Hence, our upper bound is best possible.  We show, \ref{Kdim1}, that if $R$ has Krull dimension\footnote{In this paper Krull dimension refers to the maximal length of chains of prime ideals.} $1$ and $\Gamma(R)_\infty^+$ has m-dimension $\alpha$ then $\Zg(R)$ has CB rank $\alpha$ if $\alpha$ is a limit ordinal and  $\alpha+1$ otherwise.

\begin{definition}
For $\mfrak{p}\lhd R$ a non-zero prime ideal, define the rank of $\mfrak{p}$, $\rk \mfrak{p}$, to be the supremum of the set of ordinals $\alpha$ such that $\mfrak{p}\cap S_\alpha(R)=\emptyset$ if this set is bounded and otherwise $\infty$. If $\Gamma(R)_\infty^+$ has m-dimension $\alpha$ then define $\rk 0=\alpha$.
\end{definition}

Note that if $\mfrak{p}$ is a non-zero prime ideal and $\rk \mfrak{p}\neq \infty$ then the least ordinal $\beta$ such that $S_\beta(R)\cap \mfrak{p}\neq \emptyset$ is never a limit ordinal. Thus, if $\alpha=\rk\mfrak{p}$ then $\mfrak{p}\cap S_\alpha(R)=\emptyset$.

To aid readability, when $R$ is clear from the context, we will sometimes write $S_\alpha$ for the set $S_\alpha(R)$.

\begin{lemma}
Let $\mfrak{p}\lhd R$ be a non-zero prime ideal. Then $\beta=\rk\mfrak{p}<\infty$ if and only if $\mfrak{p}R_{S_\beta}$ is generated by an irreducible element of $R_{S_\beta}$. Moreover, if $\Gamma(R)_\infty^+$ has m-dimension $\alpha$ then $\rk \mfrak{p}\leq \alpha$ for all prime ideals $\mfrak{p}\lhd R$ and, in this situation, $0$ is the unique prime ideal with rank $\alpha$.
\end{lemma}
\begin{proof}
Suppose $\beta=\rk\mfrak{p}<\infty$. Then $S_\beta\cap \mfrak{p}= \emptyset$. Hence $\mfrak{p}R_{S_\beta}$ is a proper ideal which contains an element of $S_{\beta+1}$. Therefore $\mfrak{p}R_{S_\beta}$ contains, and hence is generated by, an irreducible element of $R_{S_\beta}$. The reverse direction follows from the definition of $S_\beta$.

If $\Gamma(R)_\infty^+$ has m-dimension $\alpha$ then $S_\alpha=R\backslash\{0\}$ and hence the only prime ideal $\mfrak{p}$ with $\mfrak{p}\cap S_\alpha=\emptyset$ is the zero ideal.
\end{proof}

The next remark can be proved by ordinal induction using the definition of $S_\alpha(R)$.
\begin{remark}
Let $\alpha,\beta$ be ordinals. Viewing all rings involved as subrings of the field of fractions of $R$ we have
\[(R_{S_\alpha(R)})_{S_\beta(R_{S_\alpha(R)})}=R_{S_{\alpha+\beta}(R)} \ \ \  \text{ and } \ \ \ S_\beta(R_{S_\alpha(R)})\cap R=S_{\alpha+\beta}(R).\] 
\end{remark}
As a consequence we get the following.
\begin{remark}\label{shiftrk}
If $\mfrak{p}\lhd R$ is a non-zero prime ideal and $\mfrak{p}\cap S_\alpha=\emptyset$ then $\rk \mfrak{p}$ is equal to $\alpha+\rk \mfrak{p}R_{S_\alpha}$.
\end{remark}

\begin{definition}
Let $R$ be a Pr\"ufer domain and let $N$ be an indecomposable pure-injective $R$-module. Define
\begin{enumerate}
\item $\Ass N:=\{r\in R \st mr=0 \text{ for some } m\in N\backslash\{0\} \}$, and
\item $\Div N:=\{r\in R \st r \text{ does not divide } m \text{ for some }m\in N\}$.
\end{enumerate}
Note that, by definition, $\Ass N\cup\Div N=\Att N$.
\end{definition}
\begin{lemma}\label{AssDivNIJ}
Let $R$ be a Pr\"ufer domain and let $(I,J)$ be an admissible pair of weakly prime ideals. Then $I^\#=\Ass N(I,J)$ and $J^\#=\Div N(I,J)$.
\end{lemma}
\begin{proof} We prove $\Ass N(I,J)=I^\#$. Suppose $a\in I^\#$. There exists $r\notin I$ such that $ar\in I$. Let $m\in N(I,J)$ realise $p(I,J)$. Then $mr\neq 0$ since $r\notin I$ and $(mr)a=mra=0$ since $ra\in I$. Hence $a\in \Ass N(I,J)$.

Suppose $a\in \Ass N(I,J)$. Let $m\in N(I,J)$ be such that $ma=0$ and let $(I',J')$ be an admissible pair of weakly prime ideals such that $p(I',J')$ is the pp-type of $m$.  By \ref{equpairs}, $(I',J')\sim (I,J)$. Hence $a\in (I:r)=I'$ or $a\in I'\subseteq (I':r)=I$. Therefore $a\in I^\#$.
%
%
%
\end{proof}


\begin{remark}\label{ZgRSrk}
For any $\alpha\leq \mdim \Gamma(R)_\infty^+$, the set of $N\in\Zg(R)$ with $\rk \Ass N,\rk \Div N\geq \alpha$ is equal to the closed set $\Zg(R_{S_\alpha})$.
\end{remark}

We now take a brief detour to investigate the relationship between ranks of prime ideals of a B\'ezout domain and their CB rank as points in $\Spec^*R$, the inverse space of $\Spec R$, i.e. the space with points the prime ideals of $R$ and a basis of open sets given by $V(a):=\{\mfrak{p}\in \Spec R \st a\in \mfrak{p}\}$ for $a\in R$. It's easy to see that a prime ideal $\mfrak{p}\lhd R$ is isolated in $\Spec^*R$ if and only if $\mfrak{p}$ is maximal and $\mfrak{p}=\rad(aR)$ for some $a\in R$.

\begin{proposition}\label{specCB}
Suppose that $\Gamma(R)^+_\infty$ has m-dimension. For prime ideals $\mfrak{p}\lhd R$, $\rk\mfrak{p}$ is equal to the CB rank of $\mfrak{p}$ in $\Spec^* R$. In particular, if $\Gamma(R)^+_\infty$ has m-dimension then $\Spec^*R$ has CB rank.
\end{proposition}
\begin{proof}
One can see directly that the proposition holds when $R$ is a field. So, we suppose that $R$ is a B\'ezout domain with non-trivial value group $\Gamma(R)$. Suppose $\mfrak{p}:=\rad(aR)$ is maximal. Since $\Gamma(R)^+_\infty$ has m-dimension and $\Gamma(R)$ is non-trivial, by \ref{Lpointsdense} and \ref{irredlatvallat}, there exists an irreducible element $c\in R$ such that $a\in cR$. Since $cR$ is prime, $\mfrak{p}=\rad(aR)\subseteq cR$. Hence $cR=\mfrak{p}$. Therefore, a prime ideal $\mfrak{p}\lhd R$ is isolated in $\Spec^*R$ if and only if $\mfrak{p}=cR$ for some $c\in R$ irreducible. That is, $\mfrak{p}\lhd R$ is isolated in $\Spec^*R$ if and only if $\rk\mfrak{p}=0$. It now follows from the definition of $\rk\mfrak{p}$ that $(\Spec^*R)^{(1)}$ is equal to $\Spec^*R_S$. By definition of $S_\alpha$, $\Spec^*R_{S_\alpha}=(\spec^*R)^{(\alpha)}$ for all $\alpha\leq \mdim \Gamma(R)_\infty^+$. Moreover, if $\beta=\mdim \Gamma(R)_\infty^+$ then $R_{S_\beta}$ is a field and hence $(\Spec^*R)^{(\beta)}$ has just one point. Therefore $\mfrak{p}$ has CB rank $\alpha$ if and only if $\rk \mfrak{p}=\alpha$.
\end{proof}

On the other hand, it is often the case that $\Spec^*R$ has CB rank but $\Gamma(R)^+_\infty$ does not have m-dimension.
For instance, take a valuation domain $V$ with value group $\R$. Then $\Spec^*V$ is $T_0$ and has only 2 points. Hence $\Spec^* V$ has CB rank despite its value group being dense.

\begin{proposition}
Let $U$ be the multiplicatively closed subset of $R$ generated by those $a\in R$ such that $\rad(aR)$ is maximal. A prime ideal $\mfrak{p}\lhd R$ is isolated in $\Spec^*R$ if and only if $ \mfrak{p}\cap U\neq \emptyset$. In particular, the first Cantor-Bendixson derivative of $\Spec^*R$ is equal to $\Spec^* R_U$.
\end{proposition}
\begin{proof}
From the description of the isolated points of $\Spec^*R$ just before \ref{specCB}, if $\mfrak{p}$ is isolated then $U\cap\mfrak{p}\neq \emptyset$. Conversely, if $U\cap\mfrak{p}\neq \emptyset$ then, because $\mfrak{p}$ is prime, there exists $a\in\mfrak{p}$ such that $\rad(aR)$ is maximal. So $\mfrak{p}$ is isolated.
%
%
\end{proof}

Back with the main thread of this section, our next task is to bound the CB rank of $N\in\Zg_R$ in terms of the rank of $\Ass N$ and $\Div N$.

\begin{lemma}\label{isowkprime}
If $I\lhd R$ is a weakly prime ideal with $\rk I^\#<\infty$ then there exists $a,c\in R$ with $a\notin I$ and $ac\in I$ such that for all weakly prime $J\lhd R$ with $\rk J^\#\geq \rk I^\#$, if $a\notin J$ and $ac\in J$ then $J=I$.
\end{lemma}
\begin{proof}
Let $\mfrak{p}:=I^\#$ and $\alpha:=\rk \mfrak{p}$. Then $\mfrak{p}R_{S_\alpha}$ is principally generated, say by $c\in R$, and hence so is $\mfrak{p}R_{\mfrak{p}}$. Thus $IR_{\mfrak{p}}=acR_{\mfrak{p}}$ for some $a\in R$. Therefore, by \ref{formwkprime}, $I=(acR_{\mfrak{p}})\cap R$. Clearly, $a\notin I$ and $ac\in I$.

Suppose $a\notin J$ and $ac\in J$. Then $c\in J^\#$. Thus $c\in J^\#\cap S_{\alpha+1}$ since $cR_{S_\alpha}$ is maximal. So $\rk J^\# \leq \alpha=\rk I^\#$. If $\rk J^\#=\alpha$ and $c\in J^\#$ then $J^\#=\mfrak{p}$. Thus $J=(cbR_{\mfrak{p}})\cap R$. Now $ac\in bcR_{\mfrak{p}}$ and $a\notin bcR_{\mfrak{p}}$ implies $acR_{\mfrak{p}}=bcR_{\mfrak{p}}$ since $cR_{\mfrak{p}}$ generates $\mfrak{p}R_{\mfrak{p}}$. Therefore $\rk J^\#\geq \rk I^\#$ implies $I=J$.
\end{proof}

\begin{lemma}\label{isolated}
A point $N\in\Zg(R)$ is isolated if and only if $\rk \Ass N=\rk \Div N =0$.
\end{lemma}
\begin{proof}
%
%
Since the isolation condition holds, see \ref{CBmdimdist}, if $N$ is isolated then it is isolated by a minimal pair. So, in particular, $N\in\left(\nf{\phi}{\psi}\right)$ where $[\psi,\phi]$ is totally ordered. Hence, by \ref{movetogoodopen}, $N\in \left(\nf{x=x}{c|x}\right)$ or $N\in \left(\nf{xc=0}{x=0}\right)$ for some irreducible $c\in R$. We leave the case where $N\in \left(\nf{xc=0}{x=0}\right)$ to the reader. Suppose $N\in \left(\nf{x=x}{c|x}\right)$.  Then $c\in \Div N$. We need to show $c\in \Ass N$. Let $m\in N$ be such that $m\notin Nc$. By \cite[9.11]{PrestBluebook}, there exists a minimal pair $\sigma/\tau$ such that $m\in\sigma(N)$ and $m\notin \tau(N)$. Let $\phi:=\sigma+c|x$ and $\psi:=\tau+c|x$. Since $N$ is pp-uniserial, $m\in \phi(N)$ and $m\notin \psi(N)$. So $N\in (\nf{\phi}{\psi})$ and $\nf{\phi}{\psi}$ is a minimal pair since $\nf{\sigma}{\tau}$ is. By the proof of \ref{xxcxchain}, if $c|x\leq \phi$ then $\phi$ is of the form $c|x+xd=0$ and the map from $\Gamma(R_{\mfrak{m}})_\infty^+\rightarrow [c|x,x=x]$ which sends $dR_{\mfrak{m}}$ to $c|x+xd=0$ is a lattice isomorphism. An interval $[dR_{\mfrak{m}},deR_{\mfrak{m}}]$ in $\Gamma(R_{\mfrak{m}})_\infty^+$ for $d\neq 0$ is simple if and only if $eR_{\mfrak{m}}=cR_{\mfrak{m}}$. Therefore $\phi$ is $c|x+xdc=0$ and $\psi$ is $c|x+xd=0$ for some non-zero $d\in R_{\mfrak{m}}$. Therefore $c\in \Ass N$. Hence $\rk \Ass N=\rk \Div N=0$.

Suppose $N=N(I,J)$ where $(I,J)$ is an admissible pair of weakly prime ideals and $\rk\Ass N=\rk \Div N=0$. Then $\rk I^\#=\rk\Ass N=0$ and $\rk J^\#=\rk\Div N=0$. So, by \ref{isowkprime}, there exists $a,b,c,d$ such that $a\notin I$, $ac\in I$, $b\notin J$, $bd\in J$ and for all weakly prime ideals $K\lhd R$, if $a\notin K$ and $ac\in K$ then $K=I$ and if $b\notin K$ and $bd\in K$ then $K=J$. Let
\[\mcal{W}:=\left(\nf{b|x\wedge xac=0}{xa=0+bd|x}\right).\] By \ref{toponpoint}, $N(I,J)\in \mcal{W}$.

Now, if $N'\in \mcal{W}$ then, by \ref{toponpoint}, $N'=N(I',J')$ for an admissible pair of weakly prime ideals $(I',J')$ such that $a\notin I'$, $ac\in I'$, $b\notin J'$ and $bd\in J'$. Therefore, since $\rk I^\#=\rk J^\#=0$, by \ref{isowkprime}, $I=I'$, $J=J'$ and hence $N=N(I,J)=N(I',J')=N'$ as required.
\end{proof}

\begin{cor}\label{isolnotfield}
Let $R$ be a B\'ezout domain which is not a field. If $N\in\Zg(R)$ is isolated then there exists $c\in R$ irreducible such that $\Ass N=\Div N=cR$. In particular, $\Zg(R_S)\subseteq \Zg(R)^{(1)}$.
\end{cor}

Recall that the foundation rank $\fdrk a$ of an element $a$ of an ordered set $(A,\leq)$ is defined inductively as follows: $\fdrk a=0$ if and only if there are no elements $b>a$, and, for any ordinal $\alpha$, $\fdrk a \leq \alpha$ if for all $b>a$, $\fdrk b<\alpha$.

\begin{proposition}\label{upperbound}
Suppose that $\Gamma(R)_{\infty}^+$ has m-dimension. The CB rank of $N\in \Zg(R)$ is less than or equal to the foundation rank of $(\rk \Ass N,\rk \Div N)$ in $\text{Ord}^2$ ordered coordinatewise. In particular, if $\Gamma(R)_{\infty}^+$ has m-dimension $\alpha$ then $\CBrank(\Zg(R))\leq \alpha\cdot 2$.
\end{proposition}
\begin{proof}
We show that for all $\beta\in \text{Ord}$, if $\beta$ is the foundation rank of $(\rk \Ass N,\rk \Div N)$ then $N\notin \Zg(R)^{(\beta+1)}$. If $\beta=0$ then $\rk\Ass N=\rk \Div N=0$. By \ref{isolated}, $N$ is isolated in $\Zg(R)$ and so $N\notin \Zg(R)^{(1)}$. Suppose we have proved the statement for all $\beta'<\beta$. We want to show that $N$ is isolated in $\Zg(R)^{(\beta)}$. By \ref{ALLNIJ} and \ref{AssDivNIJ}, $N=N(I,J)$ for an admissible pair $(I,J)$ of weakly prime ideals with $I^\#=\Ass N$ and $J^\#=\Div N$. Let $a,c\in R$ (respectively $b,d\in R$) be such that $a\notin I$ and $ac\in I$ (respectively $b\notin J$ and $bd\in J$), and, such that for all weakly prime $K\lhd R$ with $\rk K^\#\geq \rk I^\#$ (respectively $\rk K^\#\geq \rk J^\#$), if $a\notin K$ and $ac\in K$ (respectively $b\notin K$ and $bd\in K$) then $K=I$ (respectively $K=J$). Such $a,b,c,d\in R$ exist by \ref{isowkprime}.

We now show that if $N\in \Zg(R)^{(\beta)}$ then it is isolated by
\[\mcal{W}:= \left(\nf{b|x\wedge xac=0}{xa=0+bd|x}\right)\]
in $\Zg(R)^{(\beta)}$. Suppose $N'\in\mcal{W}$. Then, by \ref{toponpoint}, $N'=N(K,L)$ for $(K,L)$ an admissible pair of weakly prime ideals with $a\notin K$, $ac\in K$, $b\notin L$ and $bd\in L$. Then, by definition of $a,b,c,d$, either
\begin{enumerate}
\item $I=K$ and $J=L$,
\item $\rk K^\#<\rk I^\#$ and $J=L$,
\item $K=I$ and $\rk J^\#<\rk L^\#$, or
\item $\rk K^\#<\rk I^\#$ and $\rk J^\#<\rk L^\#$.
\end{enumerate}
In the final $3$ cases the foundation rank of $(\rk \Ass N',\rk \Div N')=(\rk K^\#,\rk L^\#)$ is strictly less that the foundation rank of $(\rk \Ass N,\rk \Div N)$. So, $N'\notin \Zg(R)^{(\beta)}$. In the first case $N=N(I,J)=N(K,L)=N'$.
\end{proof}

The upper bound for the CB rank of the Ziegler spectrum in the previous proposition also holds for Pr\"ufer domains by \ref{BeztoPruf}. By \ref{GammappRmdimLB}, the m-dimension of $\Gamma(R)_\infty^+$ is a lower bound for the m-dimension of $\pp_R^1$. Hence, by \ref{CBmdimdist}, also a lower bound for $\CBrank(\Zg(R))$.

\begin{cor}\label{CBrankbounds}
Let $R$ be a Pr\"ufer domain. If $\Gamma(R)_\infty^+$ has m-dimension $\alpha$ then \[\alpha\leq \CBrank(\Zg(R))\leq \alpha\cdot 2.\]
\end{cor}

We now reinterpret a result of Puninski, which computes the CB ranks of Ziegler spectra of valuation domains, to show that the upper bound in the previous corollary is the best possible bound on the CB rank of $\Zg(R)$ in terms of the m-dimension of $\Gamma(R)_{\infty}^+$.

Let $S$ be a ring. For each ordinal $\alpha$, we define an ideal $J(\alpha)(S)\lhd S$ by ordinal induction. Define $J(0)(S)$ to be the Jacobson radical of $S$ and $J(\alpha+1)(S):=\bigcap_{n\in\N} J(\alpha)(S)^n$. If $\beta$ is a limit ordinal then let $J(\beta)(S):=\bigcap_{\alpha<\beta}J(\alpha)(S)$.

\begin{theorem}\cite[2.1 \& 3.6]{PunCB}
Let $V$ be a valuation domain. The Ziegler spectrum of $V$ has CB rank if and only if there exists an ordinal $\alpha$ such that $J(\alpha)(V)=0$. Moreover, the CB rank of $\Zg(V)$ is $\alpha\cdot 2$ if and only if $\alpha$ is the least ordinal such that $J(\alpha)(V)=0$.
\end{theorem}

We will show that the least ordinal $\alpha$ such that $J(\alpha)(V)=0$ is equal to $\mdim \Gamma(V)_\infty^+$. In order to do this we need a lemma.

\begin{lemma}\label{BetterdefJalpha}
Let $V$ be a valuation domain and $\alpha\in \Ord$. Then $J(\alpha)(V)$ is prime and $J(\alpha+1)(V)=J(1)(V_{J(\alpha)(V)})$ as subsets of the field of fractions of $V$.
\end{lemma}
\begin{proof}
The intersection of a set of prime ideals of a valuation domain $V$ is a prime ideal and, \cite[Ch II 1.3(c)]{MONND}, for any ideal $I\lhd V$, $\bigcap_{n\in \N} I^n$ is prime. Therefore $J(\alpha)(V)$ is prime for all $\alpha\in \Ord$. Since $V$ is a valuation domain, for all prime ideals $\mfrak{p}\lhd V$, $\mfrak{p}V_{\mfrak{p}}=\mfrak{p}$ as subsets of the field of fractions of $V$. Therefore for all $\alpha\in \Ord$,
\[J(\alpha+1)(V)=\bigcap_{n\in\N}J(\alpha)(V)^n=\bigcap_{n\in\N}(J(\alpha)(V)V_{J(\alpha)(V)})^n=J(1)(V_{J(\alpha)(V)}).\]
\end{proof}

\begin{cor}
Let $V$ be a valuation domain such that $\Gamma(V)_\infty^+$ has m-dimension $\alpha$. Then $\Zg(V)$ has CB rank $\alpha\cdot 2$.
\end{cor}
\begin{proof}
We show that for all $\alpha\in\Ord$, $S_\alpha(V)=V\backslash J(\alpha)(V)$. The result then follows from Puninski's result, since, by \ref{dimintermsofmult}, the m-dimension of $\Gamma(V)_\infty^+$ is the least ordinal $\alpha$ such that $S_\alpha(V)=V\backslash\{0\}$.

By definition, $S_0(V)$ is the set of units of $V$ and, since $V$ is local, $J(0)(V)$ is the set of non-units of $V$. So $S_0(V)=V\backslash J(0)(V)$.

If $J(0)(V)^2=J(0)(V)$ then either $J(0)(V)=0$ or $J(0)(V)$ is not finitely generated. In either case, $J(0)(V)=J(1)(V)$ and $S_0(V)=S_1(V)$ since $V$ has no irreducible elements. Therefore $V\backslash J(1)(V)=S_1(V)$.

If $J(0)(V)^2\neq J(0)(V)$ then $J(0)(V)=cV$ for some non-zero $c\in V$. Since $J(0)(V)$ is maximal, $c$ is irreducible. Then $a\notin J(1)(V)$ if and only if $a\in V\backslash c^nV$ for some $n\in\N$ if and only if $a=c^mu$ for some $m\in \N_0$ and $u$ a unit. Since $J(0)(V)$ is maximal, $c$ is the unique irreducible element of $V$ and hence $a\in S_1(V)$ if and only if $a=uc^m$ for some $m\in \N_0$ and unit $u$. Therefore $V\backslash J(1)(V)=S_1(V)$.

So we have proved that $S_\alpha(V)=V\backslash J(\alpha)(V)$ for $\alpha =1$. We now proceed by ordinal induction.
It follows from the definitions of $J(\alpha)(V)$ and $S_\alpha(V)$ that if $\lambda$ is a limit ordinal and $S_\alpha(V)=V\backslash J(\alpha)(V)$ for all $\alpha<\lambda$ then $S_\lambda(V)=V\backslash J(\lambda)(V)$.

Suppose $S_\alpha(V)=V\backslash J(\alpha)(V)$. Then
\[S_{\alpha+1}(V)=S_{1}(V_{S_\alpha(V)})\cap V=[V_{S_\alpha(V)}\backslash J(1)(V_{S_\alpha(V)})]\cap V=V\backslash J(\alpha +1)(V).\]
The first equality is by definition of $S_{\alpha+1}(V)$, the second follows from the claim with $\alpha=1$ and the third follows from \ref{BetterdefJalpha}.
%
\end{proof}

We now calculate the CB rank of $\Zg(R)$ in the extreme opposite case, that is, when $R$ is a Pr\"ufer domain of Krull dimension $1$. The first case is straight forward.

\begin{remark}\label{mdim1}
Let $R$ be a Pr\"ufer domain. If $\mdim\Gamma(R)_\infty^+=1$ then $R$ is a Dedekind domain which is not a field. Hence, if $\mdim\Gamma(R)_\infty^+=1$ then $\Zg(R)$ has CB rank $2$.
\end{remark}
\begin{proof}
We just need to show that $R$ is Noetherian. Since $\mdim\Gamma(R)_\infty^+=1$, $\Gamma(R)$ is generated by elements $a\in \Gamma(R)^+$ such that $[0,a]$ is a simple interval. Thus, every multiplication prime filter of $\Gamma(R)_\infty^+$ is either $\{\infty\}$ or contains such an element. In the second case, the filter is of the form $\{b\in\Gamma(R)_\infty^+ \st b\geq a\}$. It follows that all prime ideals of $R$ are finitely generated and hence that $R$ is Noetherian. It is well known, see for instance \cite[5.2.4]{PSL}, that the Ziegler spectra of Dedekind domains have CB rank $2$.
\end{proof}

\begin{remark}
If $R$ is a Pr\"ufer domain of Krull dimension $1$ such that $\Gamma(R)^+_\infty$ has m-dimension then for all maximal ideals $\mfrak{m}\lhd R$,  $R_{\mfrak{m}}$ is a discrete valuation domain.
\end{remark}
\begin{proof}
Since $\Gamma(R_\mfrak{m})_\infty^+$ is a quotient of $\Gamma(R)_\infty^+$, it has m-dimension and hence $\mfrak{m}R_{\mfrak{m}}$ is a principal ideal. Therefore, since $R$ has Krull dimension $1$, $R_\mfrak{m}$ is a discrete valuation domain.
\end{proof}

\begin{lemma}\label{ptskr1}
Let $R$ be a Pr\"ufer domain of Krull dimension $1$ and such that $\Gamma(R)^+_\infty$ has m-dimension. For all $N\in\Zg(R)$, either
\begin{enumerate}
\item $\Ass N=\Div N$;
\item $\Ass N=0$ and hence $N=N(0,\Div N)$; or
\item $\Div N=0$ and hence $N=N(\Ass N,0)$.
\end{enumerate}
\end{lemma}
\begin{proof}
Since $\Ass N,\Div N$ are prime ideals and $\Ass N\cup \Div N=\Att N$ is a prime ideal, it follows directly from the assumption that $R$ has Krull dimension $1$ that either $\Ass N=\Div N$, $\Ass N=0$ or $\Div N=0$.
%

Suppose $N$ is an indecomposable pure-injective $R$-module and $\Ass N=0$. Then $N=N(0,I)$. Now either $\Div N=0$ or $\mfrak{m}:=\Div N$ is a maximal ideal. In the first case $I=0$. In the second case, since $R_\mfrak{m}$ is a discrete valuation domain, all non-zero ideals of $R_\mfrak{m}$ are of the form $c^nR_\mfrak{m}$ for some $c\in R$. Therefore $I=c^nR_\mfrak{m}\cap R$. It is easy to check that $(I:c^{n-1})=cR_\mfrak{m}\cap R=\mfrak{m}$. Therefore, by \ref{equpairs}, $N(0,I)\cong N(0,\mfrak{m})$ as required.
\end{proof}

For each ordinal $\alpha$ with $\text{m-dim}\Gamma(R)_\infty^+\geq \alpha+1$, define
\[\Delta_\alpha(R):=\{N(\mfrak{m},0),N(0,\mfrak{m}) \st \mfrak{m}\lhd R \text{  is a maximal ideal and } \rk \,\mfrak{m}=\alpha\}.\]

The next remark follows from \ref{shiftrk} and \ref{ZgRSrk}.
\begin{remark}\label{shiftDelta}
Let $R$ be a B\'ezout domain of Krull dimension $1$. For all ordinals $\alpha$ with $\mdim \Gamma(R)_\infty^+\geq \alpha+1$, viewing $\Zg(R_{S_\alpha})$ as a subspace of $\Zg(R)$,
\[\Delta_\alpha(R)=\Delta_0(R_{S_\alpha}).\]
\end{remark}

We now compute the first two Cantor-Bendixson derivatives of $\Zg(R)$. This will form the base case for an inductive description of the higher Cantor-Bendixson derivatives in \ref{CBrankpointkr1}.

\begin{lemma}\label{basecase}
Let $R$ be a B\'ezout domain of Krull dimension $1$ and such that $\Gamma(R)_\infty^+$ has m-dimension.
\begin{itemize}
\item [$(1)_{0,0}$] If $\mdim \Gamma(R)_\infty^+\geq 1$ then
\[\Zg(R)^{(1)}=\Zg(R_{S_1})\cup\Delta_0(R).\]
\item [$(1)_{0,1}$] If $\mdim \Gamma(R)_\infty^+\geq 2$ then $\Zg(R)^{(2)}=\Zg(R_{S_1})^{(1)}$.
\end{itemize}
\end{lemma}
\begin{proof}
\medskip

\noindent
$(1)_{0,0}$
By \ref{isolated}, $N\in\Zg(R)$ is isolated if and only if $\rk\Ass N=\rk \Div N=0$. By \ref{ZgRSrk}, $N\in\Zg(R_{S_1})$ if and only if $\rk \Ass N>0$ and $\rk \Div N>0$. Since $\mdim \Gamma(R)_\infty^+\geq 1$, $\rk 0>0$. So, by \ref{ptskr1}, $N\in \Zg(R)^{(1)}\backslash\Zg(R_{S_1})$ if and only if $N=N(0,\mfrak{m})$ or $N=N(\mfrak{m},0)$ for some maximal ideal $\mfrak{m}$ with $\rk\mfrak{m}=0$.

\medskip

\noindent
$(1)_{0,1}$
Using $(1)_{0,0}$, it is enough to show that all points in $\Delta_0(R)$ are isolated in $\Zg(R)^{(1)}$ and that if a point is isolated in $\Zg(R_{S_1})$ then it is isolated in $\Zg(R)^{(1)}$.

Suppose $\mfrak{m}\lhd R$ is a maximal ideal with $\rk\mfrak{m}=0$. There exists $c\in R$ such that $cR=\mfrak{m}$. Now $N\in \left(\nf{x=x}{c|x}\right)$ if and only if $c\in \Div N$ and hence $\Div N=\mfrak{m}$. Therefore, by \ref{ptskr1}, $N\in\left(\nf{x=x}{c|x}\right)$ if and only if either $\Ass N=\Div N =\mfrak{m}$ or $N=N(0,\mfrak{m})$. So $\left(\nf{x=x}{c|x}\right)$ isolates $N(0,\mfrak{m})$ in $\Zg(R)^{(1)}$.

Now suppose that $N\in \Zg(R_{S_1})$ is isolated in $\Zg(R_{S_1})$. Since $\mdim \Gamma(R)_\infty^+\geq 2$, $\mdim \Gamma(R_{S_1})_\infty^+\geq 1$ and hence $R_{S_1}$ is not a field. By \ref{isolnotfield}, $\Ass N=\Div N= cR_{S_1}\cap R$ for some $c\in R$ which is irreducible as an element of $R_{S_1}$. Let $\mcal{U}$ be an open subset of $\Zg(R)$ such that $\mcal{U}\cap \Zg(R_{S_1})=\{N\}$. Let
\[\mcal{W}:=\mcal{U}\cap \left(\nf{x=x}{c|x}\right)\cap\left(\nf{xc=0}{x=0}\right).\] For all $M\in\Delta_0(R)$, either $\Ass M=0$ or $\Div M=0$. So, in particular, either $c\notin \Ass M$ or $c\notin \Div M$. Therefore $\mcal{W}\cap \Delta_0(R)=\emptyset$. Hence $\mcal{W}\cap \Zg(R)^{(1)}=\{N\}$.
\end{proof}

We will now compute the Cantor-Bendixson derivatives of $\Zg(R)$.

\begin{proposition}\label{CBrankpointkr1} Let $R$ be a B\'ezout domain of Krull dimension $1$ such that $\Gamma(R)_\infty^+$ has m-dimension.
Let $\lambda$ be a limit ordinal and $n\in\N_0$.
\begin{itemize}
\item [$(0)_\lambda$] If $\mdim \Gamma(R)_\infty^+\geq \lambda$ then $\Zg(R)^{(\lambda)}=\Zg(R_{S_\lambda})$.
\item [$(1)_{\lambda,n}$] If $\mdim \Gamma(R)_\infty^+\geq \lambda+n+1$ then \[\Zg(R)^{(\lambda+n+1)}=\Zg(R_{S_{\lambda+n}})^{(1)}=\Zg(R_{S_{\lambda+n+1}})\cup\Delta_{\lambda+n}(R).\]

\end{itemize}

\end{proposition}

\begin{proof}
We prove the proposition by induction on $\lambda$. First we prove the base case $\lambda=0$. The statement $(0)_0$ is trivial. The first equality of $(1)_{0,0}$ is trivial and we have already proved the second equality in \ref{basecase}. For all $n\in \N$, the second equality in $(1)_{0,n}$ follows from the first using $(1)_{0,0}$ because $\mdim \Gamma(R)_\infty^+\geq n+1$ implies $\mdim \Gamma(R_{S_{\lambda+n}})_\infty^+\geq 1$ and hence
\[\Zg(R_{S_{\lambda+n}})^{(1)}=\Zg(R_{S_{\lambda+n+1}})\cup\Delta_{0}(R_{S_{\lambda+n}})=\Zg(R_{S_{\lambda+n+1}})\cup\Delta_{\lambda+n}(R).\]
The second equality follows from \ref{shiftDelta}.
So we just need to prove the first inequality in $(1)_{0,n}$ for all $n\in\N$. We have already proved the first equality of $(1)_{0,1}$ in \ref{basecase}. Suppose that $(1)_{0,n}$ is true and $\mdim \Gamma(R)_\infty^+\geq (n+1)+1$.
Then
\[\Zg(R)^{((n+1)+1)}=(\Zg(R)^{(n+1)})'=(\Zg(R_{S_n})^{(1)})'=\Zg(R_{S_n})^{(1+1)}=\Zg(R_{S_{n+1}})^{(1)}.\]
The second equality is an application of $(1)_{0,n}$ and the forth equality is an application of $(1)_{0,1}$ applied to $R_{S_n}$ (note that $\mdim\Gamma(R_{S_{n}})_\infty^+\geq 1+1$). So, by induction on $n$ we have proved the first equality in $(1)_{0,n}$ for all $n\in\N_0$ and we have seen that this implies the second equality in $(1)_{0,n}$. This completes the base case.

We now prove $(0)_\lambda$ and $(1)_{\lambda,n}$ for all limit ordinals $\lambda$ and $n\in\N_0$ by induction on $\lambda$. Let $\lambda$ be a limit ordinal and suppose we have shown  $(0)_{\alpha}$, $(1)_{\alpha,n}$ for all limit ordinals $\alpha<\lambda$ and $n\in\N_0$.
In order to prove $(0)_\lambda$, suppose $\mdim \Gamma(R)_\infty^+\geq \lambda$. We split into two cases.

\noindent
\textbf{Case 1:} $\lambda = \gamma+\omega$ for some limit ordinal $\gamma$.
\smallskip

\noindent
Using $(1)_{\gamma,n}$ for $n\in\N_0$,
\[\Zg(R)^{(\lambda)}=\bigcap_{n\in\N}\Zg(R)^{(\gamma+n+1)}=\bigcap_{n\in\N}\Zg(R_{S_{\gamma+n}})^{(1)}\subseteq \Zg(R_{S_\lambda}).\]
We have already seen that if $\mdim \Gamma(R)_\infty^+\geq 1$ then $\Zg(R_S)\subseteq \Zg(R)^{(1)}$. Therefore, since $\mdim \Gamma(R_{S_{\gamma+n}})\geq 1$, $\Zg(R_{S_{\gamma+n+1}})\subseteq \Zg(R_{S_{\gamma+n}})^{(1)}$. Hence $\Zg(R_{S_\lambda})\subseteq \Zg(R_{S_{\gamma+n}})^{(n)}$ for all $n\in\N_0$. So the final inclusion displayed above is actually an equality as required.

\noindent
\textbf{Case 2:} The limit ordinals $\beta<\lambda$ are cofinal in $\lambda$.
\smallskip

\noindent
By the induction hypothesis,
\[\Zg(R)^{(\lambda)}=\bigcap_{\beta<\lambda}\Zg(R)^{(\beta)}=\bigcap_{\beta<\lambda}\Zg(R_{S_\beta})=\Zg(R_{S_\lambda}).\]
This proves $(0)_\lambda$.

In order to prove $(1)_{\lambda,n}$, suppose $\mdim \Gamma(R)_\infty^+\geq \lambda+n+1$. By $(0)_\lambda$,  $\Zg(R)^{(\lambda+n+1)}=\Zg(R_{S_\lambda})^{(n+1)}$. Moreover $\mdim\Gamma(R_{S_\lambda})_\infty^+\geq n+1$. So $(1)_{0,n}$ implies
\[\Zg(R_{S_\lambda})^{(n+1)}=\Zg(R_{S_{\lambda+n}})^{(1)}=\Zg(R_{S_{\lambda+n+1}})\cup\Delta_{n}(R_{S_\lambda})=\Zg(R_{S_{\lambda+n+1}})\cup\Delta_{\lambda+n}(R).\] So we have proved $(1)_{\lambda,n}$ for all $n\in\N_0$.
\end{proof}

\begin{theorem}\label{Kdim1}
Let $R$ be a Pr\"ufer domain of Krull dimension $1$. If $\Gamma(R)_\infty^+$ has m-dimension $\alpha$  then the CB rank of $\Zg(R)$ is $\alpha$ if $\alpha$ is a limit ordinal and $\alpha+1$ otherwise.
\end{theorem}
\begin{proof} By \ref{BeztoPruf}, combined with the Jaffard-Kaplansky-Ohm theorem, it is enough to prove the result for B\'ezout domains.
Suppose $\alpha$ is a limit ordinal. Then, by \ref{CBrankpointkr1}, $\Zg(R)^{(\alpha)}=\Zg(R_{S_\alpha})$. Since $\Gamma(R)_\infty^+$ has m-dimension $\alpha$, $R_{S_\alpha}$ is a field and hence $\Zg(R_{S_\alpha})$ has just one point. Therefore $\Zg(R)$ has CB rank $\alpha$.
Suppose $\alpha=\lambda+n+1$ where $n\in\N_0$ and $\lambda$ is a limit ordinal. Then, by \ref{CBrankpointkr1}, $\Zg(R)^{(\lambda+n+1)}=\Zg(R_{S_{\lambda+n}})^{(1)}$. Since $\Gamma(R)_\infty^+$ has m-dimension $\alpha$, $\Gamma(R_{S_{\lambda+n}})_\infty^+$ has m-dimension $1$. So, by \ref{mdim1}, $\Zg(R_{S_{\lambda+n}})^{(1)}$ has CB rank $1$. Thus $\Zg(R)$ has CB rank $\alpha+1$.
\end{proof}

A B\'ezout domain has Krull dimension $1$ if and only if $\Gamma$ is non-trivial and all multiplication prime filters of $\Gamma^+$ are maximal. In the next subsection, for each ordinal $\alpha$, we construct an $\ell$-group $\Gamma$ such that $\Gamma_\infty^+$ has m-dimension $\alpha$ and all multiplication prime filters of $\Gamma^+$ are maximal. As a consequence, \ref{Krdim1specifiedmdim}, we show that for each ordinal $\alpha$ there exists a B\'ezout domain $R$ with Krull dimension $1$ such that $\Gamma(R)_\infty^+$ has m-dimension $\alpha$.  Therefore we get the following corollary to \ref{Kdim1}.

\begin{cor}For all ordinals $\alpha$ which are not of the form $\lambda+1$ where $\lambda$ is a limit ordinal, there exists a B\'ezout domain whose Ziegler spectrum has CB rank $\alpha$.
\end{cor}

\section{Lattice ordered abelian groups of continuous functions}\label{Sctsfn}

We consider the $\ell$-group $C(X,\Z)$ of continuous functions from a Boolean space $X$ to $\Z$ equipped with the discrete topology. We show that all multiplication prime filters of $C(X,\Z)^+$ are maximal. Thus any B\'ezout domain with value group $C(X,\Z)$ has Krull dimension $1$. We show that $X$ has CB rank $\alpha$ if and only if $C(X,\Z)^+_\infty$ has m-dimension $\alpha+1$. In order to construct an $\ell$-group $\Gamma$ such that all multiplication prime filters of $\Gamma^+$ are maximal and $\Gamma_\infty^+$ has specified m-dimension $\alpha$, we take a Boolean space with CB rank $\alpha$ and take $\Gamma$ to be the convex $\ell$-subgroup of $C(X,\Z)$ of all $f\in C(X,\Z)$ with $f(x)=0$ for all $x\in X$ of CB rank $\alpha$. 

A space is \textbf{Boolean} if it is Hausdorff, compact and totally disconnected. Recall that the clopen sets of a Boolean space form a basis of open sets. There is a close connection between Boolean algebras and Boolean spaces given by Stone's duality theorem which gives a contravariant equivalence of categories between the category of Boolean spaces with continuous maps and the category of Boolean algebras. On objects this duality takes a Boolean space $X$ to $\mcal{K}(X)$ its Boolean algebra of clopen sets. In the other direction, it takes a Boolean algebra $B$ to $\text{Prim}(B)$ which is the topological space with points the prime filters of $B$ and basis of open sets of the form $V(b):=\{p\in \text{Prim}(B) \st b\in p\}$. The map \[X\rightarrow \textnormal{Prim}(\mcal{K}(X)) \ \ \ x\mapsto \mcal{P}_x:=\{K\in \mcal{K}(X)\st x\in K\}\] is a homeomorphism with inverse given by taking a prime filter $\mcal{F}\in \textnormal{Prim}(\mcal{K}(X))$ to the unique point of $X$ in $\bigcap_{\mcal{U}\in \mcal{F}}\mcal{U}$. 

For $f\in C(X,\Z)$, define \[\Supp f:=\{x\in X\st f(x)\neq 0\}.\]

\begin{remark}\label{bspfact}
\noindent
\begin{enumerate}
\item For all $f\in C(X,\Z)$, the image of $f$ is finite and for each $n\in \Z$, $f^{-1}(n)$ is clopen. In particular, $\supp f$ is clopen.
\item If $W_1,\ldots, W_l$ are a partition of $X$ into clopen sets and $a_1,\ldots,a_l\in \Z$ then the function $f$ defined by $f(x):=a_i$ for all $x\in W_i$ is continuous.
\item If $K$ is a clopen subset of $X$ then define
\[\chi_K(x):=\left\{
           \begin{array}{ll}
             1, & \hbox{if $x\in K$;} \\
             0, & \hbox{otherwise.}
           \end{array}
         \right.
\] The function $\chi_K$ is continuous.
\end{enumerate}
\end{remark}

\begin{definition}
For $f\in C(X,\Z)^+$, define $f'\in C(X,\Z)$ by setting $f'(x)=f(x)$ if $f(x)\leq 1$ and $f'(x)=1$ otherwise.
\end{definition}

It follows from \ref{bspfact} that $f'$ is continuous. Note that, if $f\in C(X,\Z)^+$ then $\Supp f'=\Supp f$. Hence, if $f,g\in C(X,\Z)^+$ and $\Supp f=\Supp g$ then $f'=g'$.

We record some properties of $f'$ and $\Supp$ which follow easily from their definitions.

\begin{remark}\label{Suppandf'}
Let $f,g\in C(X,\Z)^+$. Then
\begin{enumerate}
\item there exists $n\in\N$ such that $f'\leq f\leq nf'$;
\item $\Supp f'\subseteq \Supp g'$ if and only if $f'\leq g'$;
\item $\Supp (f\wedge g) =\Supp(f)\cap \Supp(g)$; and
\item $\Supp (f\vee g)=\Supp(f)\cup\Supp(g)$.
\end{enumerate}
\end{remark}

%
%
\begin{lemma}\label{primefilters}
The map
\[\mcal{F}\mapsto B(\mcal{F}):=\{\Supp f\st f\in\mcal{F}\}\] is a bijection between multiplication prime filters of $C(X,\Z)^+$ and prime filters of $\mcal{K}(X)$. In particular, the multiplication prime filters of $C(X,\Z)^+$ are exactly the filters of the form
\[\mcal{F}_x:=\{f\in C(X,\Z)^+ \st f(x)>0\}\] where $x\in X$. Moreover, they are all maximal.
%
%
\end{lemma}
\begin{proof}
We start by showing that if $\mcal{F}$ is a multiplication prime filter of $C(X,\Z)^+$ then $B(\mcal{F})$ is a prime filter of $\mcal{K}(X)$. Since for all $f\in C(X,\Z)^+$, $f\in\mcal{F}$ if and only if $f'\in\mcal{F}$ and $\supp f=\supp f'$,
\[B(\mcal{F}):=\{\Supp f' \st f'\in \mcal{F}\}.\] So, it follows from \ref{Suppandf'} $(2)$ and $(3)$ that $B(\mcal{F})$ is a filter. If $\supp f=\emptyset$ then $f=0$. Therefore $B(\mcal{F})$ is a proper filter. It now follows from \ref{Suppandf'} $(4)$ that $B(\mcal{F})$ is a prime filter.

That $\mcal{F}_x$ is a multiplication prime filter follows easily from it's definition. Hence, it follows from Stone duality that the map is surjective. In order to show that the map is injective, suppose that $\mcal{F}$ is a multiplication prime filter of $C(X,\Z)^+$ and that $f\in C(X,\Z)^+$ with $\Supp f\in B(\mcal{F})$. There exists $g\in \mcal{F}$ such that $\supp g=\supp f$. Therefore $\supp f'=\supp g'$ and hence $f'=g'$. So $g'=f'\in \mcal{F}$ and hence $f\in\mcal{F}$ as required.
\end{proof}

In this section we are only interested in m-dimension. So we will write $C_\alpha$ for the subgroup $C_{\two,\alpha}(C(X,\Z))$ of $C(X,\Z)$.

\begin{lemma}\label{C1ctsfn1}
If $f\in C(X,\Z)^+$ is such that $[0,f]$ is a simple interval then there exists an isolated point $y\in X$ such that $f(y)=1$ and $f(x)=0$ for all $x\neq y$. In particular $C_1$ is the set of $f\in C(X,\Z)$ with $\Supp f\subseteq X\backslash X^{(1)}$.
\end{lemma}

\begin{proof} Suppose that $f\in C(X,\Z)^+$ is such that $[0,f]$ is a simple interval.
Then $f\geq f'>0$. So $f=f'$. Suppose that $\Supp f$ contains two distinct elements $x,y$. Since $\mcal{K}(X)$ is a basis and $X$ is Hausdorff, there exists a clopen set $V\subseteq \Supp f$ such that $x\in V$ and $y\notin V$. Then $\chi_V< f$ because $V$ is a proper subset of $\Supp f$ and $0<\chi_V$ because $V\neq \emptyset$. Thus, if $[0,f]$ is a simple interval then $\Supp f$ is a one point clopen set $\{y\}$, $f(y)=1$ and $f(x)=0$ for all $x\neq y$.

If $f\in C_1$ then, by \ref{Ldimlgroup}, $f=\sum_{i=1}^mf_i-\sum_{i=m+1}^nf_i$ where $f_i\in C(X,\Z)^+$ and $[0,f_i]$ and  is a simple interval for $1\leq i\leq n$. It follows from the first part of the lemma that $\Supp f\subseteq\bigcup_{i=1}^n\Supp f_i$ is a subset of $X\backslash X^{(1)}$. Conversely, suppose $\Supp f\subseteq X\backslash X^{(1)}$. Then $\Supp f$ contains only isolated points. So, since $\Supp f$ is compact it is finite. Let $\Supp f=\{x_1,\ldots,x_n\}$. For each $1\leq i\leq n$, define $f_i\in C^+(X,\Z)$ by setting $f_i(x_i)=1$ and $f_i(x)=0$ for all $x\neq x_i$. Now each $f_i\in C_1$ and there exists $m_i\in\Z$ for $1\leq i\leq n$ such that $f=\sum_{i=1}^nm_if_i$. So $f\in C_1$ as required.
\end{proof}

\begin{remark}\label{resclosedset}
Let $X$ a Boolean space and $Y$ a closed subset. The map
\[C(X,\Z)\rightarrow C(Y,\Z), \ \ \ f\mapsto f|_Y,\] induced by restriction to $Y$, is a surjective homomorphism of $\ell$-groups.
\end{remark}
\begin{proof}
It is easy to see that the map is an $\ell$-group homomorphism. So we just need to show it is surjective.

Note $Y$ is Boolean. By Stone duality and since epimorphisms in the category of Boolean algebras are surjective, the map taking $K\in\mcal{K}(X)$ to $K\cap Y\in \mcal{K}(Y)$ is surjective (see \cite[7.6 (b)]{Booleanalg1} for a direct proof). It follows easily that any partition of $Y$ into clopen sets is the intersection with $Y$ of a partition of $X$ into clopen sets.

Take $f\in C(Y,\Z)$. Let $f(Y)=\{a_1,\ldots,a_n\}$ and let $Y_i=f^{-1}(a_i)$ for $1\leq i\leq n$. By the argument in the previous paragraph, there exists clopen sets $X_i$ for $1\leq i\leq n$ which partition $X$ and such that $X_i\cap Y=Y_i$. Define $g\in C(X,\Z)$ by setting $g(x)=a_i$ for $x\in X_i$. Then $g$ is continuous and $g|_Y=f$ as required.
\end{proof}


\begin{lemma}\label{C1ctsfn2}
The restriction map
\[f\in C(X,\Z) \mapsto f|_{X^{(1)}}\in C(X^{(1)},\Z)\] is surjective and has kernel $C_1$. Therefore $C(X,\Z)/C_1$ is isomorphic to $C(X^{(1)},\Z)$.
\end{lemma}
\begin{proof}
The restriction map from $C(X,\Z)$ to $C(X^{(1)},\Z)$ is a surjective $\ell$-group homomorphism since $X^{(1)}$ is closed by \ref{resclosedset}.

For $f\in C(X,\Z)$, $f|_{X^{(1)}}=0$ if and only if  $\Supp f\subseteq X\backslash X^{(1)}$. So, by \ref{C1ctsfn1}, the kernel of the restriction map from $C(X,\Z)$ to $C(X^{(1)},\Z)$ is $C_1$.
\end{proof}

\begin{lemma}\label{mdimCXZ}
For all ordinals $\alpha$, $C(X^{(\alpha)},\Z)\cong C(X,\Z)/C_{\alpha}$. In particular, if $X$ has CB rank $\beta$ then $C(X,\Z)_\infty^+$ has m-dimension $\beta+1$.
\end{lemma}
\begin{proof}
We show, by induction on ordinals, that the restriction map $f\mapsto f|_{X^{(\alpha)}}$ has kernel $C_\alpha$. Since $X^{(\alpha)}$ is closed, this map is surjective and hence this will prove the first claim.

For $\alpha=0$, $C(X^{(\alpha)},\Z)=C(X,\Z)$ and $C_0=\{0\}$. The successor step follows from \ref{C1ctsfn2}.


Suppose $\lambda$ is a limit ordinal. Then $f\in C_\lambda$ if and only if $f\in C_\alpha$ for some $\alpha<\lambda$. Therefore $f\in C_\lambda$ if and only if $f|_{X^{(\alpha)}}=0$ for some $\alpha<\lambda$. Thus $f\in C_\lambda$ implies $f|_{X^{(\lambda)}}=0$. Conversely, suppose $f|_{X^{(\lambda)}}=0$. Then $\Supp f\subseteq \bigcup_{\alpha<\lambda}(X\backslash X^{(\alpha)})$. Since $\Supp f$ is closed and $X$ is compact, $\Supp f$ is compact. Therefore $\Supp f\subseteq X\backslash X^{(\alpha)}$ for some $\alpha<\lambda$. Hence $f|_{X^{(\alpha)}}=0$ for some $\alpha<\lambda$. Thus $f\in C_\alpha$ for some $\alpha<\lambda$ as required.

By \ref{dimforgroup2}, the m-dimension of $C(X,\Z)_\infty^+$ is the least ordinal $\alpha$ such that $C_\alpha=C(X,\Z)$. So, the m-dimension of $C(X,\Z)_\infty^+$ is the least ordinal $\alpha$ such that $C(X,\Z)/C_{\alpha}=\{0\}$. Therefore, the m-dimension of $C(X,\Z)_\infty^+$ is the least ordinal $\alpha$ such that $X^{(\alpha)}=\emptyset$. Since $X$ is compact, $\alpha=\beta+1$ for some ordinal $\beta$ and the CB rank of $X$ is $\beta$ as required.
\end{proof}

Unfortunately this means that the m-dimension of $C(X,\Z)_\infty^+$ is never a limit ordinal. The problem is that, since $X$ is compact, if $X$ has CB rank $\alpha$ then $X^{(\alpha)}$ is non-empty. 

\begin{definition}
Let $X$ be a Boolean space of CB rank $\alpha$. Define
\begin{enumerate}[(i)]
\item $C^-(X,\Z)$ to be the $\ell$-subgroup of $f\in C(X,\Z)$ such that $f(x)=0$ for all $x\in X^{(\alpha)}$, and
\item $\mcal{K}^{-}(X)$ be the lattice of clopen sets of $X$ which do not contain any point of CB rank $\alpha$.
\end{enumerate}
\end{definition}

Note that $\mcal{K}^-(X)$ is unlikely to be a bounded lattice.

%

\begin{lemma}
The prime filters of $\mcal{K}^-(X)$ are of the form
\[\mcal{P}^-_x:=\{K\in \mcal{K}^-(X) \st x\in K\}\] where $x\in X\backslash X^{(\alpha)}$.
Moreover, they are all maximal as filters in $\mcal{K}^-(X)$.
\end{lemma}
\begin{proof}
Since $\mcal{K}^-(X)$ is distributive, the complement of a prime filter is an ideal. Let $\mcal{P}$ be a prime filter in $\mcal{K}^-(X)$. Let $\mcal{P}'$ be the filter generated by $\mcal{P}$ in $\mcal{K}(X)$. Then $\mcal{P}'\cap \mcal{K}^-(X)=\mcal{P}$. Let $\mcal{I}$ be $\mcal{K}^-(X)\backslash \mcal{P}$. Since $\mcal{K}^-(X)$ is an ideal of $\mcal{K}(X)$, $\mcal{I}$ is an ideal of $\mcal{K}(X)$.  Therefore $\mcal{P}'\cap \mcal{I}\subseteq \mcal{P}\cap \mcal{I}=\emptyset$. One can now use Zorn's lemma to produce a prime filter $\mcal{F}$ in $\mcal{K}(X)$ such that $\mcal{F}\supseteq \mcal{P}$ and $\mcal{F}\cap \mcal{I}=\emptyset$. So $\mcal{F}\cap \mcal{K}^-(X)=\mcal{P}$. By Stone duality, there exists $x\in X$ such that $\mcal{F}$ is the set of $K\in\mcal{K}(X)$ such that $x\in K$. If $x\in X^{(\alpha)}$ then $\mcal{F}\cap \mcal{K}^-(X)=\emptyset$. So $x\in X\backslash X^{(\alpha)}$ as required.

It remains to show that if $x,y\in X\backslash X^{(\alpha)}$ are non-equal then $\mcal{F}^-_x$ and $\mcal{F}^-_y$ are incomparable. The set of points of CB rank $\alpha$ is closed. So, since $\mcal{K}(X)$ is a basis for $X$, there exists $U\in\mcal{K}(X)$ such that $x,y\in U$ and $U\subseteq X\backslash X^{(\alpha)}$. Since $X$ is Hausdorff, there exists $K\in \mcal{K}(X)$ such that $x\in K$ and $y\notin K$. Then $x\in K\cap U$ and $y\notin K\cap U$. Thus $K\cap U\in \mcal{F}^-_x$ and $K\cap U\notin \mcal{F}^-_y$. Therefore $\mcal{F}^-_x$ is not a subset of $\mcal{F}^-_y$. Symmetrically, one can show that $\mcal{F}^-_y$ is not a subset of $\mcal{F}^-_x$. So $\mcal{F}^-_x$ is maximal for all $x\in X\backslash X^{(\alpha)}$.
\end{proof}

We show that every multiplication prime filter of $C^-(X,\Z)^+$ is the intersection of a multiplication prime filter of $C(X,\Z)$ with $C^-(X,\Z)^+$.

\begin{lemma}\label{removexfilt}
Let $X$ be a Boolean space of CB rank $\alpha$. The map
\[\mcal{F}\mapsto B(\mcal{F}):=\{\Supp f\st f\in\mcal{F}\}\] is a bijection between multiplication prime filters of $C^-(X,\Z)^+$ and prime filters of $\mcal{K}^-(X)$. In particular, the multiplication prime filters of $C^-(X,\Z)^+$ are exactly the filters of the form
\[\mcal{F}^-_x:=\{f\in C^-(X,\Z)^+ \st f(x)>0\}\] where $x\in X\backslash X^{(\alpha)}$. Moreover, they are all maximal.
%
%
\end{lemma}
\begin{proof}
Note that $f\in C^-(X,\Z)$ if and only if $f'\in C^-(X,\Z)$. That $B(\mcal{F})$ is a prime filter of $\mcal{K}^-(X)$ for $\mcal{F}$ a multiplication prime filter of $C^-(X,\Z)$ and that the map is injective follows exactly as in \ref{primefilters}. That the map is surjective and the rest of the lemma follows from \ref{removexfilt}.
\end{proof}

\begin{lemma}\label{C-mdim}
Let $X$ be a Boolean space with CB rank $\alpha$. Then $C^-(X,\Z)_\infty^+$ has m-dimension $\alpha$.
%
\end{lemma}
\begin{proof}
%
%
%
%
%
%
%
%
For $\alpha\in\Ord$, let $C_\alpha^-$ denote the subgroup $C_{\two,\alpha}(C^-(X,\Z))$ of $C^-(X,\Z)$. It is easy to see that $C^-(X,\Z)$ is a convex subgroup of $C(X,\Z)$. Thus, for all $f\in C^-(X,\Z)^+$, $[0,f]$ is simple in $C^-(X,\Z)$ if and only if it is simple in $C(X,\Z)$. Thus, by \ref{Ldimlgroup}, $C_1\cap C^-(X,\Z)=C_1^-$. So it follows from \ref{C1ctsfn2} that the restriction map \[f\in C^-(X,\Z)\mapsto f|_{X^{(1)}}\in C^-(X^{(1)},\Z) \] is surjective with kernel $C_1^-$. As in \ref{mdimCXZ} for $C(X,\Z)$, it follows by ordinal induction that $C^-(X^{(\beta)},\Z)$ is isomorphic to $C^-(X,\Z)/C^-_\beta$. Therefore, by \ref{dimforgroup2}, the m-dimension of $C^-(X,\Z)_\infty^+$ is the least ordinal $\beta$ such that $C^-(X^{(\beta)},\Z)=\{0\}$. For any non-empty Boolean space $Y$ with CB rank, $C^-(Y,\Z)=\{0\}$ if and only if $Y$ has CB rank $0$. Thus $C^-(X^{(\beta)},\Z)=\{0\}$ if and only if $\beta$ is the CB rank of $X$. So, the CB rank of $X$ is equal to the m-dimension of $C^-(X,\Z)_\infty^+$.
%
%
%
\end{proof}

\begin{proposition}\label{Krdim1specifiedmdim}
For each ordinal $\alpha$, there exists a B\'ezout domain $R$ with Krull dimension $1$ such that $\Gamma(R)_\infty^+$ has m-dimension $\alpha$.
\end{proposition}
\begin{proof} Let $X$ be a Boolean space with CB rank $\alpha$. That such a space exists is well known; for instance, see \cite[4.5]{MayerPierce}, where it is shown that the ordinal $\omega^\alpha+1$ with the interval topology has CB rank $\alpha$.  By the Jaffard-Kaplansky-Ohm theorem, there exists a B\'ezout domain $R$ with $\Gamma(R)\cong C^-(X,\Z)$. By \ref{C-mdim}, $\Gamma(R)_\infty^+$ has m-dimension $\alpha$. The non-zero prime ideals of $R$ are in order preserving bijection with the multiplication prime filters of $\Gamma(R)^+$ which are all maximal by \ref{removexfilt}. So $R$ has Krull dimension $1$ as required.
\end{proof}

\bibliographystyle{alpha}
\bibliography{Transferthm}

\end{document}